\documentclass[11pt]{article}
\usepackage{amsmath}
\usepackage{amssymb}
\usepackage{graphicx}
\usepackage{esint}

\makeatletter

\usepackage{amsthm}\usepackage{amsfonts}\usepackage{amscd}\usepackage{mathrsfs}\usepackage{bm}\usepackage{enumerate}

\newcommand{\D}{\displaystyle}

\newcommand{\al}{\alpha}
\newcommand{\e}{\epsilon}

\newcommand{\ot}{\overline{\theta}}
\newcommand{\dt}{\Delta\theta}

\newcommand{\ol}[1]{\overline{#1}}
\newcommand{\ho}[1]{\widehat{\overline{#1}}}

\newcommand{\myref}[1]{(\ref{#1})}

\newtheorem{theorem}{\textbf{Theorem}}[section]
\newtheorem{lemma}{\textbf{Lemma}}[section]
\newtheorem{remark}{\textbf{Remark}}[section]
\newtheorem{proposition}{\textbf{Proposition}}[section]

\newtheorem{corollary}{\textbf{Corollary}}[section]

\makeatother

\begin{document}

\title{Convergence of a data-driven time-frequency analysis method}

\author{Thomas Y. Hou%
\thanks{Applied and Comput. Math, Caltech, Pasadena, CA 91125. \textit{Email:
hou@cms.caltech.edu.}%
} \and Zuoqiang Shi%
\thanks{Mathematical Sciences Center, Tsinghua University, Beijing, China,
100084. \textit{Email: zqshi@math.tsinghua.edu.cn.}%
} \and Peyman Tavallali %
\thanks{Applied and Comput. Math, Caltech, Pasadena, CA 91125. \textit{Email:
ptavalla@caltech.edu.}%
}}

\maketitle

\begin{abstract}
In a recent paper \cite{HS12}, 
Hou and Shi introduced a new adaptive data analysis 
method to analyze nonlinear and non-stationary data. The main
idea is to look for the sparsest representation of multiscale
data within the largest possible dictionary consisting of
intrinsic mode functions of the form $\{ a(t) \cos(\theta(t))\}$, 
where $a \in V(\theta)$, $V(\theta)$ consists of the functions smoother
than $\cos(\theta(t))$ and $\theta'\ge 0$. This problem was 
formulated as a nonlinear $L^0$ optimization problem and
an iterative nonlinear matching pursuit method was proposed to
solve this nonlinear optimization problem. In this paper, we
prove the convergence of this nonlinear matching pursuit
method under some sparsity assumption on the signal. We consider
both well-resolved and sparse sampled signals. In the case without
noise, we prove that our method gives exact recovery of the original
signal.
\end{abstract}
\section{Introduction}

Developing a truly adaptive data analysis method is important for
our understanding of many natural phenomena. Although a number of
effective data analysis methods such as the Fourier transform or
windowed Fourier transform have been developed, these methods use
pre-determined
basis and are mostly used to process linear and stationary data.
Applications of these methods to nonlinear and nonstationary
data tend to give many unphysical harmonic modes. To overcome these
limitations of the traditional techniques, time-frequency analysis has
been developed by representing a signal with a joint function of both
time and frequency \cite{Flandrin99}. The recent advances of
wavelet analysis have led to the development of several powerful
wavelet-based time-frequency analysis techniques
\cite{JP90,Daub92,OW04,Mallat09}. But they still cannot remove
the artificial harmonics completely and do not give satisfactory 
results for nonlinear signals.

Another important approach in the time-frequency analysis is to
study instantaneous frequency of a signal. Some of the
pioneering work in this area was due to Van der Pol
\cite{VdP46} and Gabor \cite{Gabor46}, who introduced
the so-called Analytic Signal (AS) method that uses the
Hilbert transform to determine instantaneous frequency of
a signal. However, this method works mostly for monocomponent
signals in which the number of zero-crossings is equal to the
number of local extrema \cite{Boashash92c}. There were other attempts
to define instantaneous frequency such as the zero-crossing method
\cite{Rice44,Shekel53,Meville83} and the Wigner-Ville distribution
method \cite{Boashash92c,LWB93,QC96,Flandrin99,LT96,Picinbono97}.
Most of these methods suffer from various limitations. For example,
the zero-crossing method cannot be applied to study a signal with
multiple components and is sensitive to noise. On the other hand,
the methods based on the Wigner-Ville distribution suffer from the
interference between different components.

More substantial progress has been made recently with the introduction 
of the Empirical Mode Decomposition (EMD) method \cite{Huang98}. The 
EMD method decomposes a signal into a collection of intrinsic mode 
functions (IMFs) sequentially through a sifting process. On the other 
hand, since the EMD method relies on the information of local extrema 
of a signal, it is unstable to noise perturbation. Recently, an ensemble 
EMD method (EEMD) was proposed to make it more stable to noise 
perturbation \cite{WH09}. But some fundamental issues remain unresolved.

Inspired by EMD/EEMD and the recently developed compressive
sensing theory \cite{CRT06a,Candes-Tao06,Donoho06,BDE09}, Hou and
Shi proposed a data-driven time-frequency analysis method in a 
recent paper \cite{HS12}. The main idea of this method is to 
look for the sparsest decomposition of a signal over the largest 
possible dictionary consisting of the intrinsic mode functions (IMFs). 
The dictionary is chosen to be: 
\begin{eqnarray}
\mathcal{D}=\left\{ a\cos\theta:\; a\in V(\theta),\;\theta'\in V(\theta),\mbox{and}\;\theta'(t)\ge0,\forall t\in\mathbb{R}\right\} ,\label{dic-D}
\end{eqnarray}
 where $V(\theta)$ is a collection of all the functions that are
smoother than $\cos\theta(t)$. In general, it is most effective to
construct $V(\theta)$ as an overcomplete Fourier basis in the 
$\theta$-space. For periodic signals, we can simply choose $V(\theta)$ as
the standard Fourier basis in the $\theta$-space.
Then the problem can be reformulated as a nonlinear version of the
$L^0$ minimization problem. 
\begin{eqnarray}
\begin{array}{rcc}
\vspace{-2mm}P: & \mbox{Minimize} & M\\
\vspace{2mm} & {\scriptstyle (a_{k})_{1\le k\le M},(\theta_{k})_{1\le k\le M}}\\
 & \mbox{Subject to:} & \left\{ \begin{array}{l}
f=\sum_{k=1}^{M}a_{k}\cos\theta_{k},\\
a_{k}\cos\theta_{k}\in\mathcal{D},\;\quad k=1,\cdots,M.
\end{array}\right.
\end{array}
\end{eqnarray}
The constraint $f=\sum_{k=1}^{M}a_{k}\cos\theta_{k}$ can be replaced by 
an inequality when the signal is polluted by noise. 
This kind of optimization problem is known to be very challenging to solve
since both $a_k$ and $\theta_k$ are unknown.
Inspired by matching pursuit \cite{MZ93,TG07}, Hou and Shi \cite{HS12} proposed a nonlinear matching pursuit method to solve this nonlinear 
optimization problem. The basic idea is to decompose the signal 
sequentially into two parts, the mean plus a modulated oscillatory part with zero mean:
\begin{eqnarray}
  f=a_0+a_1\cos\theta ,
\end{eqnarray}
where the mean $a_0$, the envelope $a_1$, and the phase function 
$\theta$ are all unknown. We call $a_1\cos\theta$ an Intrinsic Mode 
Function (IMF). After this decomposition is completed, we can
treat $a_0$ as a new signal and repeat this process 
until the residual is small enough.

The objective of this paper is to analyze the convergence of the 
data-driven time-frequency analysis method 
proposed by Hou and Shi in \cite{HS12} for periodic signals.
We assume that the signal $f$ has a sparse representation over the Fourier basis in the $\theta$-space for some unknown $\theta$. The main objective
of our data-driven time-frequency analysis is to design an iterative 
algorithm to find such $\theta$. With a
given approximate phase function $\theta^n$, we solve 
a $l^1$ minimization problem to obtain the Fourier coefficients of $f$ in 
the $\theta^n$-space:
\begin{eqnarray}
\label{l1-intro}
  \min_x \|x\|_1,\quad \mbox{subject to}\quad \Phi_{\theta^n}x=f ,
\end{eqnarray}
where each column of matrix $\Phi_{\theta^n}$ is a Fourier basis in the
 $\theta^n$-space. We then use this coefficient $x$ to update $\theta^n$, 
and repeat this process until it converges. 

When the signal has sufficiently well-resolved samples, 
the $l^1$ optimization problem \myref{l1-intro} can 
be solved very efficiently by interpolation and Fast Fourier Transform (FFT). 
In this case, the constraint $\Phi_{\theta^n}x=f$ becomes a well-posed 
deterministic linear system
provided that the coefficient matrix $\Phi_{\theta^n}$ is invertible.
The linear optimization problem is then reduced to solving this linear system.
Since the matrix $\Phi_{\theta^n}$ consists of the 
Fourier basis in the $\theta^n$-space, the corresponding linear system 
can be solved approximately by first interpolating $f$ to a uniform mesh 
in the $\theta^n$-space and then applying FFT. This gives rise to a very 
efficient algorithm with complexity of order $O(N\log(N))$, where $N$ is the
number of sample points of the signal. Details of this algorithm will be given in 
Section \ref{section-resolve}.

Our first result is for well-resolved periodic signals of the form
$f(t) = a_0(t) + a_1(t) \cos\theta(t)$. We ignore the interpolation
error and assume that $f(t)$ is given for all $t \in [0,T]$. 
We further assume that the instantaneous frequency 
$\theta'(t)$ has a sparse representation in the Fourier basis 
in the physical space given by
$\left\{ e^{i2k\pi t/T}, \;\; |k| \leq M_0 \right\}$,
$a_0$ and $a_1$ have a sparse representation in the Fourier basis 
in the normalized $\theta$-space given by
$\left\{ e^{i2k\pi \bar{\theta}}, \;\; |k| \leq M_1 \right\}$,
where $\bar{\theta} = \frac{\theta(t) - \theta(0)}{\theta(T) - \theta(0)}$ is the normalized phase function. Then we can prove that the 
iterative algorithm will converge to the exact solution under some
mild scale separation assumption on the signal. More precisely,
if the initial guess of $\theta$ satisfies 
\begin{eqnarray}
\|\mathcal{F}\left(\left(\theta^{0}-\theta\right)'\right)\|_{1}\le\pi M_{0}/2,\;\label{cond-initial-intro}
\end{eqnarray}
where $\mathcal{F}$ is the Fourier transform in the physical space,
 then there exists $\eta_0>0$ such that
\begin{eqnarray}
\|\mathcal{F}\left(\left(\theta^{m+1}-\theta\right)'\right)\|_{1}\le\frac{1}{2}\left\Vert \mathcal{F}\left(\left(\theta^{m}-\theta\right)'\right)\right\Vert _{1},
\end{eqnarray}
provided that $L\ge \eta_0$, where $\eta_0$ is a constant determined by $M_0$, $M_1$ and $L=\frac{\theta(T)-\theta(0)}{2\pi}$.
We remark that $1/L$ is a measure of the smallest scale
of the signal. The scales of $a_0$, $a_1$, and $\theta$ are measured by 
$1/M_1$ and $1/M_0$ respectively. The requirement $L\ge \eta_0$ is 
actually a mathematical formulation of the scale separation property. 

The key idea of the proof is to estimate the decay rate of the coefficients over the 
Fourier basis in the $\theta^n$-space, where $\theta^n$ is the approximate phase 
function in each step. We show that the Fourier coefficients of the signal in the 
$\theta^n$-space have a very fast decay as long as that $\theta^n$ is a smooth function. 
Using this estimate, we can show that the error of the phase function in each step is
a contraction and the iteration converges to the exact solution.

In many problems, a signal may not has an exact sparse representation. A more general 
setting is that the Fourier coefficients of $a_0$, $a_1$, and $\theta'$ 
decay according to some power law as the wave number increases.
We can prove that in 
this case, our method will converge to an approximate solution with an error 
determined by the truncated error of $a_0$, $a_1$ and $\theta'$. 
The detailed analysis will be presented in Section \ref{section-resolve-stable}.

For signals with sparse samples, we can also prove similar convergence results with 
an extra condition on the matrix $\Phi_{\theta^n}$. In this case, we need to use the $l^1$ minimization even with periodic signals. Suppose
$S$ is the largest number such that $\delta_{3S}(\Phi_{\theta^{n}})+3\delta_{4S}(\Phi_{\theta^{n}})<2$.
Under the same sparsity assumption on the instantaneous frequency, the mean
and the envelope as before, we can prove that
there exist $\eta_L>0,\; \eta_S>0$, such that
\begin{eqnarray}
\|\mathcal{F}\left(\left(\theta^{m+1}-\theta\right)'\right)\|_{1}\le \frac{1}{2}\left\Vert \mathcal{F}\left(\left(\theta^{m}-\theta\right)'\right)\right\Vert _{1},
\end{eqnarray}
provided that $L\ge \eta_L$ and $S\ge \eta_S$.
Here the columns of the matrix consist of the Fourier basis in the 
$\theta^n$-space, $\delta_S(A)$ is the $S$-restricted isometry constant of 
matrix $A$ given in \cite{CT06a}.
Further, we show that if the 
sample points $\{t_j\}_{j=1}^{N_s}$ are selected at random from a set of 
uniformly distributed points $\{t_l\}_{j=1}^{N_f}$, the condition 
$\delta_{3S}(\Phi_{\theta^{n}})+3\delta_{4S}(\Phi_{\theta^{n}})<2$
holds with an overwhelming probability provided that
$ S\le C N_s/(\max(\ol{\theta}'(\log N_b )^6) $
and $N_f\ge \max \{ C\|\widehat{\ol{\theta}'}\|_{1}N_b,2M_{0}\}$,
 where $N_s$ is the number of the samples, $N_b$ is the number of the basis.
If $M_0=0$, which implies that $\ol{\theta}'=1$, then the
above result is reduced to the well-known theorem for the standard Fourier 
basis in \cite{CRT06b}. 

The rest of the paper is organized as follows. In Section 2, we 
establish the convergence and stability of our method for 
well-resolved signals. In Section 3, we propose an algorithm for signals 
with sparse samples and prove its convergence and stability. In Section 4,
some numerical results are presented to demonstrate the performance of
the algorithm and confirm the theoretical results. Some concluding 
remarks are made in Section 5.

\section{Well resolved periodic signal}
\label{section-resolve}

In this section, we will analyze the convergence and stability of the 
algorithm proposed in \cite{HS12} for well-resolved signals. By 
well-resolved signals, we mean that that these signals are measured over 
a uniform grid and can be interpolated to any grid with very little loss 
of accuracy. In the analysis, we assume that the signal is periodic in 
the sample domain. Without loss of generality, we assume that the 
signal $f$ is periodic over $[0,1]$.

In order to make this paper self-contained, we state the 
algorithm proposed in \cite{HS12}. The signal $f$ is given over a
uniform grid $t_j = j/N$ for $j=0,...,N-1$.
\begin{itemize}
\item $\theta_{k}^{0}=\theta_{0},\; n=0$. 
\item Step 1: Interpolate $r_{k-1}$ from the uniform grid in the time domain to a uniform mesh in the $\theta_{k}^{n}$-coordinate
to get $r_{\theta_{k}^{n}}^{k-1}$ and compute the Fourier transform
$\widehat{r}_{\theta_{k}^{n}}^{k-1}$: 
\begin{eqnarray}
r_{\theta_{k}^{n},\, j}^{k-1}=\mbox{Interpolate}\;\left(t_{i},r^{k-1},\theta_{k,\, j}^{n}\right),
\end{eqnarray}
 where $\theta_{k,\, j}^{n},\; j=0,\cdots,N-1$ are uniformly distributed
in the $\theta_{k}^{n}$-coordinate,i.e. $\theta_{k,\, j}^{n}=2\pi L_{\theta_{k}^{n}}\; j/N$.
Apply the Fourier transform to $r_{\theta_{k}^{n}}^{k-1}$ as follows:
\begin{eqnarray}
\widehat{r}_{\theta_{k}^{n}}^{k-1}(\omega)=\frac{1}{N}\sum_{j=1}^{N}r_{\theta^{n},\, j}^{k-1}e^{-i2\pi\omega\ol{\theta}_{k,\, j}^{n}},\quad\omega=-N/2+1,\cdots,N/2,
\end{eqnarray}
 where $\ol{\theta}_{k,\, j}^{n}=\frac{\theta_{k,\, j}^{n}-\theta_{k,\,0}^{n}}{2\pi L_{\theta_{k}^{n}}}$.
\item Step 2: Apply a cutoff function to the Fourier Transform of $r_{\theta_{k}^{n}}^{k-1}$
to compute $a$ and $b$ on the mesh in the $\theta_{k}^{n}$-coordinate,
denoted by $a_{\theta_{k}^{n}}$ and $b_{\theta_{k}^{n}}$: 
\begin{eqnarray}
a_{\theta_{k}^{n}} & = & \mathcal{F}_{\theta_{k}^{n}}^{-1}\left[\left(\widehat{r}_{\theta_{k}^{n}}^{k-1}\left(\omega+L_{\theta_{k}^{n}}\right)+\widehat{r}_{\theta_{k}^{n}}^{k-1}\left(\omega-L_{\theta_{k}^{n}}\right)\right)\cdot\chi\left(\omega/L_{\theta_{k}^{n}}\right)\right],\\
b_{\theta_{k}^{n}} & = & \mathcal{F}_{\theta_{k}^{n}}^{-1}\left[-i\cdot\left(\widehat{r}_{\theta_{k}^{n}}^{k-1}\left(\omega+L_{\theta_{k}^{n}}\right)-\widehat{r}_{\theta_{k}^{n}}^{k-1}\left(\omega-L_{\theta_{k}^{n}}\right)\right)\cdot\chi\left(\omega/L_{\theta_{k}^{n}}\right)\right],
\end{eqnarray}
where  $\mathcal{F}^{-1}$ is the inverse Fourier transform defined in the
$\theta_{k}^{n}$ coordinate: 
\begin{eqnarray}
\mathcal{F}_{\theta_k^n}^{-1}\left(\widehat{r}_{\theta_{k}^{n}}^{k-1}\right)=\frac{1}{N}\sum_{\omega=-N/2+1}^{N/2}\widehat{r}_{\theta_{k}^{n}}^{k-1}e^{i2\pi\omega\ol{\theta}_{k,\, j}^{n}},\quad j=0,\cdots,N-1,
\end{eqnarray}
 and $\chi$ is the cutoff function, 
\begin{eqnarray}
\chi(\omega)=\left\{ \begin{array}{cl}
1, & -1/2<\omega<1/2,\\
0, & \mbox{otherwise.}
\end{array}\right.\label{cutoff-jump}
\end{eqnarray}
 
\item Step 3: Interpolate $a_{\theta_{k}^{n}}$ and $b_{\theta_{k}^{n}}$
back to the uniform mesh in the time domain: 
\begin{eqnarray}
a_{k}^{n+1} & = & \mbox{Interpolate}\;\left(\theta_{k,\, j}^{n},a_{\theta_{k}^{n}},t_{i}\right),\quad i=0,\cdots,N-1,\\
b_{k}^{n+1} & = & \mbox{Interpolate}\;\left(\theta_{k,\, j}^{n},b_{\theta_{k}^{n}},t_{i}\right),\quad i=0,\cdots,N-1,.
\end{eqnarray}

\item Step 4: Update $\theta^{n}$ in the $t$-coordinate:
\begin{eqnarray}
\dt'=P_{V_{M_{0}}}\left(\frac{d}{dt}\left(\arctan\left(\frac{b_{k}^{n+1}}{a_{k}^{n+1}}\right)\right)\right),\;\dt(t)=\int_{0}^{t}\dt'(s)ds,\quad\theta^{n+1}=\theta^{n}+\beta\dt,
\nonumber
\end{eqnarray}
 where $\beta\in[0,1]$ is chosen to make sure that $\theta_{k}^{n+1}$
is monotonically increasing: 
\begin{eqnarray}
\beta=\max\left\{ \alpha\in[0,1]:\frac{d}{dt}\left(\theta_{k}^{n}+\alpha\dt\right)\ge0\right\} ,
\end{eqnarray}
 and $P_{V_{M_{0}}}$ is the projection operator to the space 
$V_{M_{0}}=\mbox{span}\left\{ e^{i2k\pi t/T},k=-M_{0},\cdots,0,\cdots,M_{0}\right\}$ and $M_0$ is chosen 
{\it a priori}.

\item Step 5: If $\|\theta_{k}^{n+1}-\theta_{k}^{n}\|_{2}<\epsilon_{0}$,
stop. Otherwise, set $n=n+1$ and go to Step 1. 
\end{itemize}

In the previous paper \cite{HS12}, we demonstrated that this algorithm works very effectively for periodic signals and is stable to noise 
perturbation. In this paper, we will analyze its convergence and stability.
Our main results can be summarized as follows. For periodic signals that have an exact sparsity structure, we can prove that the above algorithm will converge to the exact decomposition. For periodic signals that have an approximate sparsity structure, the above algorithm will give an approximate result withe accuracy determined by the truncated error of the signal. In the following two subsections, we will present these two results separately.

\subsection{Exact recovery}
\label{section-resolve-exact}

In this section, we consider a periodic signal $f(t)$ that has the
 following decomposition:
\begin{eqnarray}
f(t)=f_{0}(t)+f_{1}(t)\cos\theta(t),\; f_1(t)>0,\; \theta'(t)>0,\; t\in[0,1],
\end{eqnarray}
where $f_0,f_1$ and $\theta$ are the exact local mean, the envelope and
the phase function that we want to recover from the signal.

First, we introduce some notations.
 Let  $L=\frac{\theta(1)-\theta(0)}{2\pi}$ be
the number of period of the signal which is a measurement of the scale
of the signal. $\;\ol{\theta}=\frac{\theta-\theta(0)}{2\pi L}$
is the normalized phase function, which is used as a coordinate
in our numerical method and analysis. $\widehat{f}_{0,\theta}(k),\widehat{f}_{1,\theta}(k)$ are the Fourier
coefficients of $f_{0},f_{1}$ in the $\ol{\theta}$-coordinate, i.e.
\begin{eqnarray}
\widehat{f}_{0,\theta}(k)=\int_{0}^{1}f_{0}\; e^{-i2\pi k\ol{\theta}}d\ol{\theta},\quad\widehat{f}_{1,\theta}(k)=\int_{0}^{1}f_{1}\; e^{-i2\pi k\ol{\theta}}d\ol{\theta},
\end{eqnarray}
We also use the notation $\mathcal{F}_{\theta}(\cdot)$ to represent the Fourier transform in the $\theta$-space and $\mathcal{F}(\cdot)$ to represent the Fourier transform in the original $t$-coordinate.

Now we can state the theorem as follows:
\begin{theorem} \label{theorem-sparse}
Assume that the instantaneous frequency $\theta'$ is $M_0$-sparse over the Fourier basis in the physical space,
i.e. 
\begin{eqnarray}
\theta'\in V_{M_{0}}=\mbox{span}\left\{ e^{i2k\pi t/T},k=-M_{0},\cdots,1,\cdots,M_{0}\right\} .
\end{eqnarray}
Further, we assume that the
local mean $f_{0}$ and the envelope $f_{1}$ are $M_1$-sparse over the Fourier basis in the $\ol{\theta}$-space, i.e.
\begin{eqnarray}
\widehat{f}_{0,\theta}(k)=\widehat{f}_{1,\theta}(k)=0,\quad\forall|k|>M_{1}.
\end{eqnarray}
If the initial guess of $\theta^0$ satisfies 
\begin{eqnarray}
\|\mathcal{F}\left(\left(\theta^{0}-\theta\right)'\right)\|_{1}\le\pi M_{0}/2,\;\label{cond-initial}
\end{eqnarray}
 then there exist $\eta_0>0$ such that
\begin{eqnarray}
\|\mathcal{F}\left(\left(\theta^{m+1}-\theta\right)'\right)\|_{1}\le\frac{1}{2}\left\Vert \mathcal{F}\left(\left(\theta^{m}-\theta\right)'\right)\right\Vert _{1},
\end{eqnarray}
provided that $L\ge \eta_0$.
\end{theorem}

Before giving the rigorous proof, we introduce some notations for the convenience of the representation. 
Let $\theta^m$ be the approximate phase function in the $m$th step, 
and $\Delta\theta^m=\theta-\theta^m$ be the error of the phase function 
in the current step. 
Let $\widetilde{a}^m$, $\widetilde{b}^m$ be the approximate envelope functions, which are obtained by using the algorithm in Step 3. Further, we define $a^m=f_1\cos\dt^m$, $b^m=f_1\sin\dt^m$, and
$\Delta a^m=a^m-\widetilde{a}^m$ , $\Delta b^m=b^m-\widetilde{b
}^m$. The quantities $a^m$ and $b^m$ can be considered as the ``exact'' 
envelope functions at the $m$th iteration since 
$\Delta\theta^m=\arctan\left(\frac{b^m}{a^m}\right)$. Thus,
we would obtain the exact phase starting from $\theta^m$ in one 
iteration. In our analysis, we need to establish a relation 
among $\Delta a^m$, $\Delta b^m$ and $\dt^m$.

One key ingredient of the proof is to estimate the integral $\int_{0}^{1}e^{i2\pi (\omega\ot-k\ot^m)}d\ot^m$. 
Fortunately, for this type of integral, we have the following lemma.
\begin{lemma} 
\label{lemma-fft}
Suppose $\phi'(t)>0,\; t\in[0,1]$, $\phi(0)=0,\;\phi(1)=1$,
and $\psi',\phi'\in V_{M_{0}}=\mbox{span}\left\{ e^{i2k\pi t},k=-M_{0},\cdots,1,\cdots,M_{0}\right\} $.
Then we have, 
\begin{eqnarray}
\left|\int_{0}^{1}e^{i\psi}e^{-i2\pi\omega\phi}d\phi\right|\le\frac{P\left(\frac{\|\widehat{\phi'}\|_{1}}{\min\phi'},n\right)
M_{0}^{n}}{|\omega|^{n}\left(\min\phi'\right)^{n}}\sum_{j=1}^{n}(2\pi M_{0})^{-j}\|\widehat{\psi'}\|_{1}^{j} ,
\end{eqnarray}
 provided that $e^{i\psi}e^{-i2\pi\omega\phi}$ is a periodic function. Here $P(x,n)$ is a $(n-1)$th order polynomial of $x$ and the coefficients 
also depend on $n$.
\end{lemma} 
\begin{remark} Regarding the polynomial $P(x,n)$, we can get an explicit expression for small $n$. For example, when $n=2$, we have
\begin{eqnarray}
\left|\frac{d^{2}}{d\phi^{2}}e^{i\psi}\right| & = & \left|i\left(\frac{\psi''}{\phi'^{2}}-\frac{\psi'\phi''}{\phi'^{3}}+
i\frac{\psi'^{2}}{\phi'^{2}}\right)e^{i\psi}\right|\le\left|\frac{\psi''}{\phi'^{2}}\right|+\left|\frac{\psi'\phi''}{\phi'^{3}}\right|+\left|\frac{\psi'^{2}}{\phi'^{2}}\right|\nonumber \\
 & \le & \frac{\max|\psi''|}{\left(\min\phi'\right)^{2}}+\frac{\max|\psi'|\max |\phi''|}{\left(\min\phi'\right)^{3}}+\frac{\left(\max|\psi'|\right)^{2}}{\left(\min\phi'\right)^{2}}\nonumber \\
 & \le & \frac{1}{\left(\min\phi'\right)^{2}}\left[\left(1+\frac{\|\widehat{\phi'}\|_{1}}{\min\phi'}\right)2\pi M_{0}\|\widehat{\psi'}\|_{1}+\|\widehat{\psi'}\|_{1}^{2}\right],
\end{eqnarray}
where we have used $\dt,\ot\in V_{M_{0}}$ in deriving the last inequality.
Then, we have $P(x,2)=x+1$. Similarly, we can also get $P(x,3)=3x^2+4x+3$. 
\end{remark}
\begin{remark}
  Lemma \ref{lemma-fft} is valid for any $n\in \mathbb{N}$. The integral 
that we would like to estimate in Lemma \ref{lemma-fft}
is actually the Fourier
transform of $e^{i\psi}$. Since $\psi$ is a smooth function, we expect 
that the Fourier transform of $e^{i\psi}$ has a rapid decay 
for $|\omega|$ large. In Lemma \ref{lemma-fft}, we give a more 
delicate decay estimate of the Fourier transform of $e^{i\psi}$. 
Such estimate is required in our proof of Theorem \ref{theorem-sparse}.
\end{remark}
\begin{proof}
Using integration by parts, then we have 
\begin{eqnarray}
 \left|\int_{0}^{1}e^{i\psi}e^{-i2\pi\omega\phi}d\phi\right|
 = \frac{1}{|2\pi\omega|^{n}}\left|\int_{0}^{1}\frac{d^{n}(e^{i\psi})}{d\phi^{n}}e^{-i2\pi\omega\phi}d\phi\right|
 \le \frac{1}{|2\pi\omega|^{n}}\max_{t\in [0,1]}\left|\frac{d^{n}(e^{i\psi})}{d\phi^{n}}\right|\nonumber.
\end{eqnarray}
Since $e^{i\psi}e^{-i2\pi\omega\phi}$ is periodic, there is no 
contribution from the boundary terms when performing integration by parts. 
 Using the fact that, $\psi',\phi'\in V_{M_{0}}$ and $\forall g\in V_{M_{0}}$, we obtain
\begin{eqnarray}
\max_{t}|g^{(n)}(t)| & \le & \sum_{k}|(2\pi k)^{n-1}\widehat{g'}(k)|\le(2\pi M_{0})^{n-1}\sum_{k}|\widehat{g'}(k)|=(2\pi M_{0})^{n-1}\|\widehat{g'}\|_{1}.\label{control-g}
\end{eqnarray}

Direct calculations give
\begin{eqnarray}
\left|\frac{d^{n}(e^{i\psi})}{d\phi^{n}}\right|\le
\frac{P\left(\frac{\|\widehat{\phi'}\|_{1}}{\min\phi'},n\right)}
{\left(\min\phi'\right)^{n}}\sum_{j=1}^{n}(2\pi M_{0})^{n-j}\|\widehat{\psi'}\|_{1}^{j}.
\end{eqnarray}
Thus, we get 
\begin{eqnarray}
\left|\int_{0}^{1}e^{i\psi}e^{-i2\pi \omega\phi}d\phi\right| 
 & \le& \frac{P\left(\frac{\|\widehat{\phi'}\|_{1}}{\min\phi'},n\right)
M_{0}^{n}}{|\omega|^{n}\left(\min\phi'\right)^{n}}\sum_{j=1}^{n}(2\pi M_{0})^{-j}\|\widehat{\psi'}\|_{1}^{j}.
\end{eqnarray}
This proves Lemma \ref{lemma-fft}.
 \end{proof}

Now we are ready to prove Theorem \ref{theorem-sparse}.
\begin{proof} \textit{of Theorem \ref{theorem-sparse}}
 
First, we need to establish the relation between $\Delta\theta^{m+1}$ and $\Delta a^m$, $\Delta b^m$. 

Recall that $\Delta\theta^m=\arctan\left(\frac{b^m}{a^m}\right)$. Thus, 
we have 
$\widetilde{\dt}=\dt^m-\arctan\left(\frac{\widetilde{b}^m}{\widetilde{a}^m}\right)=
\arctan\left(\frac{b^m}{a^m}\right)-\arctan\left(\frac{\widetilde{b}^m}{\widetilde{a}^m}\right)$. 
Using the differential mean value theorem, we know that there exists $\xi\in[0,1]$, such that
\begin{eqnarray}
\label{est-dt}
\left|\widetilde{\dt}\right| & = & \left|\arctan\left(\frac{b^m}{a^m}\right)-
\arctan\left(\frac{\widetilde{b}^m}{\widetilde{a}^m}\right)\right|=
\left|\frac{(a^m+\xi\Delta a^m)\Delta b^m-(b^m+\xi\Delta b^m)\Delta a^m}{(a^m+\xi\Delta a^m)^{2}
+(b^m+\xi\Delta b^m)^{2}}\right|\nonumber \\
 & \le & \frac{(|a^m|+|\Delta a^m|)|\Delta b^m|+(|b^m|+|\Delta b^m|)|\Delta a^m|}
{((a^m)^{2}+(b^m)^{2})/2-((\Delta a^m)^{2}+(\Delta b^m)^{2})}\nonumber \\
 & \le & D_{1}|\Delta a^m|+D_{2}|\Delta b^m|,
\end{eqnarray}
 where 
\begin{eqnarray}
D_{1}=\max_{t}\left\{ \frac{f_1+|\Delta b^m|}{f_1^2/2-
((\Delta a^m)^{2}+(\Delta b^m)^{2})}\right\},\;
\D D_{2}=\max_{t}\left\{ \frac{f_1+|\Delta a^m|}{f_1^2/2
-((\Delta a^m)^{2}+(\Delta b^m)^{2})}\right\},\quad
\end{eqnarray}
and we have used the relations that $f_1^2=(a^m)^{2}+(b^m)^{2}$ and $|a^m|, |b^m|\le f_1$.

In the algorithm, there is another smooth process when updating $\theta$, which gives the following result for $\Delta \theta^{m+1}$,
\begin{eqnarray}
\dt^{m+1}=2\pi\Delta L^{m+1} t+\widetilde{\dt}_{p,M_{0}},
\label{delta-theta-decomp}
\end{eqnarray}
 where $\widetilde{\dt}_{p,M_{0}}=P_{V_{M_{0}}}\left(\widetilde{\dt}_{p}\right)$ is the projection of $\widetilde{\dt}_p$ 
over the space $V_{M_0}$, 
 $\widetilde{\dt}_p$ and $2\pi\Delta L^{m+1} t$ are the periodic part and 
the linear part of $\widetilde{\dt}$ respectively:
\begin{eqnarray}
\widetilde{\dt}=2\pi\Delta L^{m+1}t+\widetilde{\dt}_{p}.
\end{eqnarray}
Using \myref{delta-theta-decomp}, we can estimate $(\dt^{m+1})'$ as follows, 
\begin{eqnarray}
\label{est-dt-l1-ori}
 &  & \left\Vert \mathcal{F}\left((\Delta\theta^{m+1})'\right)\right\Vert _{1}\le2\pi\Delta L^{m+1}+\left\Vert \widehat{\widetilde{\dt}'}_{p,M_{0}}\right\Vert _{1}\le2\pi\Delta L+M_{0}\left\Vert \widehat{\widetilde{\dt}}_{p,M_{0}}\right\Vert _{1}\nonumber \\
 & \le & 2\|\widetilde{\dt}\|_{\infty}+M_{0}^{2}\left\Vert \widetilde{\dt}_{p}\right\Vert _{\infty}\le(3M_{0}^{2}+2)\|\widetilde{\dt}\|_{\infty}.
\end{eqnarray}
where  we have used the fact that
\begin{eqnarray}
2\pi|\Delta L^{m+1}| & = & |\widetilde{\dt}(1)-\widetilde{\dt}(0)|\le2\|\widetilde{\dt}\|_{\infty},\\
\left\Vert \widetilde{\dt}_{p}\right\Vert _{\infty} & = & \left\Vert \widetilde{\dt}\right\Vert _{\infty}+2\pi\Delta L\le3\left\Vert \widetilde{\dt}\right\Vert _{\infty} .
\end{eqnarray}
Combining \myref{est-dt-l1-ori} with \myref{est-dt}, we get
\begin{eqnarray}
\label{est-dt-l1-0}
\left\Vert \mathcal{F}\left((\Delta\theta^{m+1})'\right)\right\Vert _{1}\le (3M_{0}^{2}+2)\left(D_{1}\|\Delta a^m\|_\infty+D_{2}\|\Delta b^m\|_\infty\right) .
\end{eqnarray}

Next, we will establish the relation among $\Delta a^m$, $\Delta b^m$ and $\dt^m$. This can be done by estimating the Fourier coefficients of 
$\ol{a}^m$, $\ol{b}^m$ in the $\theta^m$-space.

In Appendix A, we derive the following estimates of $\Delta a^m$ and $\Delta b^m$  (see \myref{exp-error-a}, \myref{exp-error-b}),
\begin{eqnarray}
\label{est-da-0}
|\Delta a^m|
 \le 2\sum_{\frac{1}{2}L^m<k<\frac{3}{2}L^m}\left|\widehat{f}_{0,\theta^m}(k)\right|+\sum_{\frac{3}{2} L^m<k<\frac{5}{2} L^m}\left(\left|\widehat{a}_{\theta^m}(k)\right|+
\left|\widehat{b}_{\theta^m}(k)\right|\right)
+\sum_{|k|>\frac{L^m}{2}}\left|\widehat{a}_{\theta^m}(k)\right|,\quad\\
\label{est-db-0}
|\Delta b^m|\le 2\sum_{\frac{1}{2}L^m<k<\frac{3}{2}L^m}\left|\widehat{f}_{0,\theta^m}(k)\right|+\sum_{\frac{3}{2} L^m<k<\frac{5}{2} L^m}\left(\left|\widehat{a}_{\theta^m}(k)\right|
+\left|\widehat{b}_{\theta^m}(k)\right|\right)
+\sum_{|k|>\frac{L^m}{2}}\left|\widehat{b}_{\theta^m}(k)\right|,\quad
\end{eqnarray}
where $\widehat{f}_{0,\theta^m}$, $\widehat{a}^m_{\theta^m}$ and $\widehat{b}^m_{\theta^m}$ are the Fourier transform of $f_0$, $a^m$ and $b^m$ in
the $\theta^m$-space.

To obtain the desired estimates, we need to use Lemma \ref{lemma-fft} to 
estimate the Fourier coefficients of $f_0,\; a^m, b^m$ in the 
$\theta^m$-space. In an effort to make the proof concise and easy to 
follow, we defer the derivation of the estimates \myref{est-a0},
\myref{est-a} and \myref{est-b} to Appendix B.
The main results of Appendix B are summarized as follows.
As long as $\gamma=\frac{\|\mathcal{F}[(\Delta \theta^m)']\|_{1}}{2\pi M_{0}}\le 1/4$, we have
\begin{eqnarray}
\label{est-a0}
|\widehat{f}_{0,\theta^m}(\omega)| 
&\le &  C_{0}Q\left(\frac{|\omega|}{2}\right)^{-n}M_{0}^{n}M_{1}\gamma,\quad \forall |\omega|>L/2
\end{eqnarray}
\begin{eqnarray}
\label{est-a}
|\widehat{a}^m_{\theta^m}(\omega)| & \le & 4C_{0}Q\left(\frac{|\omega|}{2}\right)^{-n}M_{0}^{n}(2M_{1}+1)\gamma, \quad \forall |\omega|\ge L/2.
\end{eqnarray}
\begin{eqnarray}
\label{est-b}
|\widehat{b}^m_{\theta^m}(\omega)| & \le & 4C_{0}Q\left(\frac{|\omega|}{2}\right)^{-n}M_{0}^{n}(2M_{1}+1)\gamma, \quad \forall |\omega|\ge L/2.
\end{eqnarray}
 where 
\begin{eqnarray}
\label{def-Q}
Q=\frac{P\left(z,n\right)}{\left(\min(\ot^m)'\right)^{n}},\quad z=\frac{\|\mathcal{F}[(\ot^m)']\|_{1}}{\min(\ot^m)'},\quad  
\gamma=\frac{\|\mathcal{F}[(\Delta \theta^m)']\|_{1}}{2\pi M_{0}}.
\end{eqnarray}

Using \myref{est-da-0}-\myref{est-b} and the fact that 
$\sum_{k=1}^{\infty}k^{-n}$ converges as long as $n\ge 2$, we conclude that
\begin{eqnarray}
\label{est-da-inf}
|\Delta a^m|&\le& \Gamma_{0}Q(\al L)^{-n+1}\gamma,\\
\label{est-db-inf}
|\Delta b^m|&\le& \Gamma_{0}Q(\al L)^{-n+1}\gamma,
\end{eqnarray}
where $\Gamma_0$ is a constant that depends on $M_0,M_1$ and $n$.
It follows from \myref{est-dt-l1-0}, \myref{def-Q}, \myref{est-da-inf} and
\myref{est-db-inf} that 
\begin{eqnarray}
\label{est-dt-l1-1}
 \left\Vert \mathcal{F}\left((\Delta\theta^{m+1})'\right)\right\Vert _{1}&\le&  
\Gamma_{1}(D_1+D_2)Q(\al L)^{-n+1}\left\Vert \mathcal{F}\left((\Delta\theta^{m})'\right)\right\Vert _{1},
\end{eqnarray}
where $\Gamma_1$ is a constant that depends on $M_0,M_1$ and $n$.

To complete the proof, we need to show that there exists a constant 
$\eta_0>0$ which does not change in the iterative process, 
such that $\widetilde{\beta}=\Gamma_{1}(D_1+D_2)Q(\al L)^{-n+1}\le 1/2$ 
provided that $L\ge \eta_0$. 
This seems to be trivial, simply choosing $\eta_0=\frac{1}{\al}\left(2\Gamma_{1}(D_1+D_2)Q\right)^{1/(n-1)}$ would make $\widetilde{\beta}\le 1/2$ 
provided that $L\ge \eta_0$. The problem is that $D_1, D_2, Q, \al$ 
vary during the iteration. We need to show that they are uniformly 
bounded during the iteration.

It is relatively easy to show that $\al$ is bounded,
\begin{eqnarray}
|1-\al|=\left|1-\frac{\theta^m(1)-\theta^m(0)}{\theta(1)-\theta(0)}\right|=\left|\frac{\dt^m(1)-\dt^m(0)}{2\pi L}\right|\le\frac{\| (\dt^m)'\| _{\infty}}{2\pi L}\le\|\frac{\mathcal{F}[(\Delta\theta^{m})']\|_{1}}{2\pi L}\le\frac{M_{0}}{4L},\nonumber
\end{eqnarray}
 which implies that 
$7/8\le\al\le 9/8$,
provided that $L\ge 2M_0$ and $\gamma\le 1/4$.

It is more involved to show that $Q$ is bounded. We need to first estimate $|(\ot^m)'|$ and $\|\mathcal{F}[(\ot^m)']\|_1$,
\begin{eqnarray}
|(\ot^m)'| & = &|\ot'-(\dt^m)'/(2\pi L^m)|\ge\frac{1}{\al}\left(\ot'-\|\mathcal{F}[(\dt^m)']\|_{1}/(2\pi L)\right)\ge \frac{8}{9}\left(\ot'-\frac{M_{0}}{4L}\right)
,\label{bound-IF}
\end{eqnarray}
 and 
\begin{eqnarray}
\|\mathcal{F}[(\ot^m)']\|_{1} & = & \frac{1}{\al}\| \widehat{\ot'}-\mathcal{F}[(\dt^m)']/(2\pi L)\| _{1}\le \frac{1}{\al}\left(\|\widehat{\ot'}\|_{1}+
\|\mathcal{F}[(\dt^m)']\|_{1}/(2\pi L)\right)\nonumber \\
 & \le & \frac{8}{7}\left(\|\widehat{\ot'}\|_{1}+
M_{0}/(4L)\right),
\end{eqnarray}
where we have used the assumption that $\gamma\le \frac{1}{4}$.
If $L$ satisfies the following condition,
\begin{eqnarray}
  \label{cond-Q}
  \frac{M_0}{L}\le 2\min(\ot'),
\end{eqnarray}
then we can get
\begin{eqnarray}
|(\ot^m)'|\ge\frac{4}{9}\ot',\quad \|\mathcal{F}[(\ot^m)']\|_{1}\le \frac{12}{7}\|\widehat{\ot'}\|_{1},
\label{est-z}
\end{eqnarray}
where we have used the fact that $\min(\ot')\le \max(\ot')\le \| \widehat{\ot'}\|_{1}$.
It follows from \myref{est-z} that
the term $z$ defined in \myref{def-Q} is uniformly bounded,
\begin{eqnarray}
  z\le z_0,
\end{eqnarray}
where $z_0$ is a constant depending on $\ot'$.

Based on the above estimation of $z$, the term $Q$ in \myref{def-Q} can be bounded by a constant,
\begin{eqnarray}
  \label{est-Q}
  Q=\frac{P\left(z,n\right)}{\left(\min(\ol{\theta}^m)'\right)^{n}}\le 
\left(\frac{9}{4}\right)^n\frac{P\left(z_0,n\right)}{\left(\min\ot'\right)^{n}}=Q_0,
\end{eqnarray}
where $Q_0$ is a constant that depends on $\ot'$ and $n$.

We now proceed to bound $D_1$ and $D_2$.
Note that if $|\Delta a^m|,|\Delta b^m|\le \frac{\sqrt{2}}{4}\min f_1$, we can bound $D_1$ as follows:
\begin{eqnarray}
D_{1} & = & \max\left\{ \frac{|b^m|+|\Delta b^m|}{((a^m)^{2}+(b^m)^{2})/2-((\Delta a^m)^{2}+(\Delta b^m)^{2})}\right\} \nonumber \\
 & \le & \max\frac{|f_{1}|+|\Delta b^m|}{(f_{1})^{2}/2-((\Delta a^m)^{2}+(\Delta b^m)^{2})}\nonumber\\
 & \le & \frac{4+\sqrt{2}}{\min f_{1}}=E_0 .
\end{eqnarray}
Similarly, we can show that  $D_2\le E_0$.

It is not difficult to see that
the condition $|\Delta a^m|,|\Delta b^m|\le \frac{\sqrt{2}}{4}\min f_1$ 
is valid  if $L$ satisfies
\begin{eqnarray}
\label{cond-da}
\Gamma_{0}Q_0(7L/8)^{-n+1}\le \sqrt{2}\min f_1 ,
\end{eqnarray}
since we have
\begin{eqnarray}
\label{est-da}
|\Delta a|&\le& \Gamma_{0}Q(\al L)^{-n+1}\gamma\le \frac{1}{4}\Gamma_{0}Q_0(7L/8)^{-n+1},\\
\label{est-db}
|\Delta b|&\le& \Gamma_{0}Q(\al L)^{-n+1}\gamma\le \frac{1}{4}\Gamma_{0}Q_0(7L/8)^{-n+1},
\end{eqnarray}
where we have used 
$\al\ge 7/8$, $Q\le Q_0$, the assumption $\gamma\le \frac{1}{4}$ and the estimates \myref{est-da-0}, \myref{est-db-0}.

Finally, we derive the following estimate for the error of the instantaneous frequency,
\begin{eqnarray}
\label{est-dt-final}
 \left\Vert \mathcal{F}\left((\Delta\theta^{m+1})'\right)\right\Vert _{1}\le\beta\left\Vert \mathcal{F}\left((\Delta\theta^{m})'\right)\right\Vert _{1},
\end{eqnarray}
where $\beta=\Gamma_{1}E_0Q_0(7L/8)^{-n+1}$, $\Gamma_1$ is a constant depends on $M_0, M_1, n$, $E_0$ depends on $\min f_1$, and $Q_0$ depends on $\ot'$ and $n$.

Now, we prove that if $\gamma=\frac{\|\mathcal{F}[(\dt^m)']\|_{1}}{2\pi M_0}\le \frac{1}{4}$, then we have
\begin{eqnarray}
 \left\Vert \mathcal{F}\left((\Delta\theta^{m+1})'\right)\right\Vert _{1}\le\frac{1}{2}\left\Vert \mathcal{F}\left((\Delta\theta^{m})'\right)\right\Vert _{1},
\end{eqnarray}
as long as $L$ satisfies the following conditions
\begin{eqnarray}
  \label{cond-delta-1}
L\ge 4M_1, \quad \frac{M_0}{L} &\le& \min\left\{\frac{1}{2},2\min(\ot')\right\},\\ 
\label{cond-delta-2}
\Gamma_{0}Q_0(7L/8)^{-n+1}&\le& \sqrt{2}\min f_1,\\
\label{cond-delta-3}
\Gamma_{1}E_0Q_0(7L/8)^{n-1}&\le&\frac{1}{2}.
\end{eqnarray}
It is obvious that there exist $\eta_0>0$, such that conditions \myref{cond-delta-1}-\myref{cond-delta-3} are satisfied provided that $L\ge \eta_0$. 
Here $\eta_0$ is determined by $M_0, M_1,\ot', \min f_1$ and $n$ which does not change during the iteration process. 

By induction, it is easy to show that if initially 
\begin{eqnarray*}
\frac{\|\mathcal{F}[(\theta^0-\theta)']\|_{1}}{2\pi M_0}\le \frac{1}{4},
\end{eqnarray*}
then there exists $\eta_0>0$ which is determined by $M_0, M_1,\ot', \min f_1$ and $n$, such that 
\begin{eqnarray}
 \left\Vert \mathcal{F}\left((\Delta\theta^{m+1})'\right)\right\Vert _{1}\le\frac{1}{2}\left\Vert \mathcal{F}\left((\Delta\theta^{m})'\right)\right\Vert _{1},
\end{eqnarray}
as long as $L\ge \eta_0$. This completes the proof of Theorem
\ref{theorem-sparse}.
\end{proof}

\begin{remark}
The above proof is valid for any $n\ge 2$. Note that $\eta_0$ depends on 
$n$. Theoretically, there exists an optimal choice of $n$ to make 
$\eta_0$ the smallest. 
By carefully tracking the constants in the proof, we can show that as $n$ going to $+\infty$, $\eta_0$ tends to $\delta C(n)^{1/(n-1)} M_0$, where $\delta$ is a constant independent on $n$,
and $C(n)$ is the maximum of the coefficients of polynomial $P(x,n)$ appears in Lemma \ref{lemma-fft}. We conjecture that $C(n)^{1/(n-1)}$ is bounded for $n\ge 2$. If this is the case, then $\eta_0$ is proportional to
$M_0$.
\end{remark}

\begin{remark}
Classical time-frequency analysis methods, such as the windowed 
Fourier transform or wavelet transform, in general cannot extract the
instantaneous frequency exactly for any signal due to the uncertainty 
principle. For a single linear chirp signal without amplitude 
modulation, the Wigner-Ville distribution can extract the  exact 
instantaneous frequency, but it fails if the signal consists of 
several components. Theorem \ref{theorem-sparse} shows that our 
data-driven time-frequency analysis method has the capability to 
recover the exact instantaneous frequency for a much larger range of 
signals. 
\end{remark}

\subsection{Approximate recovery}
\label{section-resolve-stable}

If the signal does not have an exact sparsity structure in the 
$\theta$-space as required in Theorem \ref{theorem-sparse},
our method cannot reproduce the exact decomposition. But the analysis 
in this subsection shows that we can still get an approximate result and 
the accuracy is determined by the truncated error of the signal. The
main result is stated below.
\begin{theorem} 
\label{theorem-decay} Assume that the instantaneous
frequency $\theta'$, has a sparse representation,
i.e. there exists $M_{0}$, such that 
\begin{eqnarray}
\theta'(t)\in V_{M_{0}}=\mbox{span}\left\{ e^{i2k\pi t/T},k=-M_{0},\cdots,1,\cdots,M_{0}\right\} .
\end{eqnarray}
 and the Fourier coefficients of the local mean $f_{0}$ and the
envelope $f_{1}$ in the $\ol{\theta}$-space have a fast decay, i.e. 
there exists $C_{0}>0,\; p\ge 4$ such that 
\begin{eqnarray}
|\widehat{f}_{0,\theta}(k)|\le C_{0}|k|^{-p},\quad|\widehat{f}_{1,\theta}(k)|\le C_{0}|k|^{-p}.
\end{eqnarray}
Then, there exists $\eta_0>4$ such that
 if $L>\eta_0$ and the intial guess satisfies
\begin{eqnarray}
\|\mathcal{F}\left(\left(\theta^{0}-\theta\right)'\right)\|_{1}\le\pi M_{0}/2,\;\label{condition-dense}
\end{eqnarray}
then we have 
\begin{eqnarray}
\|\mathcal{F}\left(\left(\theta^{m+1}-\theta\right)'\right)\|_{1}\le \Gamma_{0}(L/4)^{-p+2}
+\frac{1}{2}\left\Vert \mathcal{F}\left(\left(\theta^{m}-\theta\right)'\right)\right\Vert _{1},
\end{eqnarray}
 where $\Gamma_0>0$ is a constant determined by $C_0$, $p$, $M_0$ and $f_1$. 
\end{theorem}
\begin{remark}
This theorem shows that our iterative method 
will converge to the exact solution up to the truncation
error determined by the scale separation property.
\end{remark}

\begin{proof}
The proof is very similar to the proof of Theorem \ref{theorem-sparse}.
The only difference is that the estimates of $\widehat{f}_{0,\theta^m}(k),\;\widehat{a}^m_{\theta^m}$
and $\widehat{b}^m_{\theta^m}$ are more complicated since they are not 
sparse in the $\ot$-space. Here we only give these key estimates.

For $\widehat{f}_{0,\theta^m}(\omega),\; \omega\ne 0$, we have 
\begin{eqnarray}
\label{est-a0-0-th2}
|\widehat{f}_{0,\theta^m}| & = & \left|\int_{0}^{1}f_{0}e^{-i2\pi\omega\ot^m}d\ot^m\right|\nonumber \\
 & = & \int_0^1 \left|\sum_{k\neq 0}\widehat{f}_{0,\theta}(k)e^{i2\pi k\ot}e^{-i2\pi\omega\ot^m}d\ot^m\right|\nonumber \\
 & = & \left|\sum_{k\neq0}\widehat{f}_{0,\theta}(k)\int_{0}^{1}e^{i2\pi(\al k-\omega)\ot^m}e^{ik\dt^m/L}d\ot^m\right|
\end{eqnarray}
where $\al =L^m/L$ and $\widehat{f}_{0,\theta}(k)$ are the Fourier coefficients of $f_0$ as a funntion of $\ot$. Note that the integral is 0
when $k=0$ and $\omega\ne 0$. Thus  we exclude the case $k=0$ 
in the above summation.
In the derivation of the last equality, we have used the relationship that 
$\ot=\theta/L=(\theta^m+\Delta \theta^m)/L=\theta^m/L+\dt^m/L=\al \ot^m+\dt^m/L$.

As in the proof of the previous theorem, we also need to use Lemma 
\ref{lemma-fft}. In the previous proof, we can choose $n$ to be any
positive integer that is greater than 2. In the current theorem, the
Fourier coefficients 
$|\widehat{f}_{0,\theta}|$ and $|\widehat{f}_{1,\theta}|$ decay 
according to some power law. To obtain the desired estimates, we 
need to take $2\le n\le p-2$. This is why we require $p\ge 4$. 

Applying Lemma \ref{lemma-fft} to 
the last equality of \myref{est-a0-0-th2}, we have
\begin{eqnarray}
\label{est-a0-th2}
|\widehat{f}_{0,\theta^m}(\omega)| & \le & \sum_{k\neq0}|\widehat{f}_{0,\theta}(k)|\left|\int_{0}^{1}e^{i2\pi(\al k-\omega)\ot^m}e^{ik\dt^m/L}d\ot^m\right|\nonumber \\
 & \le & \sum_{|k|>\frac{|\omega|}{2\al}}|\widehat{f}_{0,\theta}(k)|+
\sum_{0<|k|\le \frac{|\omega|}{2\al}}|\widehat{f}_{0,\theta}(k)|\left|\int_{0}^{1}e^{i2\pi(\al k-\omega)\ot^m}e^{ik\dt^m/L}d\ot^m\right|\nonumber \\
 & \le & C_{0}\sum_{|k|>\frac{|\omega|}{2\al}}|k|^{-p}+C_{0}\sum_{0<|k|\le|\frac{\omega|}{2\al}}\frac{QM_{0}^{n}|k|^{-p}}{|\omega-\al k|^{n}}\sum_{j=1}^{n}\left|\frac{k}{L}\right|^{j}
\left(\frac{\|\mathcal{F}[(\dt^m)']\|_{1}}{2\pi M_{0}}\right)^{j}\nonumber \\
 & \le & C_{0}\int_{|\omega|/(2\al)}^\infty x^{-p}dx + 
C_{0}Q\left(\frac{|\omega|}{2}\right)^{-n}M_{0}^{n}
\left ( \sum_{0<|k|\le|\frac{\omega|}{2\al}} |k|^{-p +n}\right)
\left (\sum_{j=1}^n\left(\gamma/L\right)^{j} \right ) \nonumber\\
 & \le & C_{0}\left(\frac{|\omega|}{2\al}\right)^{-p+1}+C_{0}Q\left(\frac{|\omega|}{2}\right)^{-n}M_{0}^{n}\gamma/L,
\end{eqnarray}
 where we have used the assumption $n \leq p-2$, $\gamma\le 1/4$, and
the fact that $L \ge 1$ is the number of the periods within the
time interval $[0,1]$. Here $C_0$ is a generic constant, 
$Q$, $z$ and $\gamma$ are defined below:
\begin{eqnarray}
\label{def-Q-decay}
Q=\frac{P\left(z,n\right)}{\left(\min(\ot^m)'\right)^{n}},\quad z=\frac{\|\mathcal{F}[(\ot^m)']\|_{1}}{\min(\ot^m)'},\quad  \gamma=\frac{\|\mathcal{F}[(\Delta \theta^m)']\|_{1}}{2\pi M_{0}}.
\end{eqnarray}

Using an argument similar to that as in the derivation of 
\myref{est-a0-th2}, we can get the desired estimates for 
$\widehat{a}^m_{\theta^m}$ and $\widehat{b}^m_{\theta^m}$ as follows:
\begin{eqnarray}
\label{est-a-th2}
|\widehat{a}^m_{\theta^m}(\omega)|\le C_{0}\left(\frac{|\omega|}{2\al}\right)^{-p+1}+Q\left|\widehat{f}_{1,\theta}(0)\right||\omega|^{-n}M_{0}^{n}\gamma+C_{0}Q\left(\frac{|\omega|}{2}\right)^{-n}M_{0}^{n}\gamma.
\end{eqnarray}
\begin{eqnarray}
\label{est-b-th2}
|\widehat{b}^m_{\theta^m}(\omega)|\le C_{0}\left(\frac{|\omega|}{2\al}\right)^{-p+1}+Q\left|\widehat{f}_{1,\theta}(0)\right||\omega|^{-n}M_{0}^{n}\gamma+C_{0}Q\left(\frac{|\omega|}{2}\right)^{-n}M_{0}^{n}\gamma.
\end{eqnarray}

The estimates \myref{est-da-0} and \myref{est-db-0} remain valid in this case. Thus we obtain upper bounds for $\Delta a^m$ and $\Delta b^m$ by 
substituting \myref{est-a-th2} and
\myref{est-b-th2} into \myref{est-da-0} and \myref{est-db-0},
\begin{eqnarray}
|\Delta a^m| & \le & \Gamma_{1}L^{-p+2}+\Gamma_{2}Q(\al L)^{-n+1}\gamma,\quad\quad\quad\label{l1-a}\\
|\Delta b^m| & \le & \Gamma_{1}L^{-p+2}+\Gamma_{2}Q(\al L)^{-n+1}\gamma,\quad\label{l1-b}
\end{eqnarray}
 where $\Gamma_{1}$ is a constant depending on $C_0$,  $\Gamma_{2}$ depends
on $p$ and $\max\left(C_{0},|\widehat{f}_{1,\theta}(0)|\right)$.

Moreover, by following the same argument we did in the proof of Theorem \ref{theorem-sparse}
, we can obtain an error estimate for the instantaneous frequency,
\begin{eqnarray}
\label{est-dt-final-th2}
 \left\Vert \mathcal{F}\left((\Delta\theta^{m+1})'\right)\right\Vert _{1}\le \Gamma_{3}E_0(L/4)^{-p+2}+\Gamma_4 E_0Q_0(7L/8)^{-n+1}\left\Vert \mathcal{F}\left((\Delta\theta^{m})'\right)\right\Vert _{1},
\end{eqnarray}
as long as $\gamma\le 1/4$ and the following conditions are satisfied
\begin{eqnarray}
  \label{cond-Q-th2}
  L\ge 2M_0,\quad \frac{M_0}{L}&\le& 2\min(\ot'),\\
\Gamma_{1}(L/4)^{-p+2}+\Gamma_{2}Q_0(7L/8)^{-n+1} & \le & \sqrt{2}\min f_1,\label{cond-da-th2}\\
\Gamma_{3}E_0(L/4)^{-p+2}+\frac{1}{4}\Gamma_{4}Q_0E_0(7L/8)^{-n+1} & \le & \frac{\pi M_{0}}{2},\quad\label{cond-ratio-th2}\\
  \label{cond-beta-th2}
  \Gamma_{4}Q_0E_0(7L/8)^{-n+1}&\le&\frac{1}{2},
\end{eqnarray}
where $\Gamma_3, \Gamma_4$ are constants that depend on $C_0, p, M_0, \min f_1$ and $\ot'$.
Using these four constraints, we can easily derive a constant $\eta_0 > 4$, such that all these conditions are satisfied provided that $L\ge \eta_0$.
On the other hand, since $L > 4$ and $n \ge 2$, \myref{cond-beta-th2}
 implies that $\Gamma_{4}Q_0E_0 \le 1/2$. This proves
\begin{eqnarray}
 \left\Vert \mathcal{F}\left((\Delta\theta^{m+1})'\right)\right\Vert _{1}\le \Gamma_{0}E_0(L/4)^{-p+2}+ \frac{1}{2}\left\Vert \mathcal{F}\left((\Delta\theta^{m})'\right)\right\Vert _{1}.
\end{eqnarray}
This completes the proof of Theorem \ref{theorem-decay}.
\end{proof}
\begin{remark}
The constraint $n\le p-2$ in the above proof can be relaxed to 
$p\ge 3$ by using a more delicate calculation. 
\end{remark}

If we further consider a more general case: the instantaneous frequency is also approximately sparse instead of exactly sparse as we assume in Theorem \ref{theorem-sparse} and \ref{theorem-decay}. 
In this case, we can prove that the iterative algorithm also converges to
an approximate result. However, we cannot apply Lemma \ref{lemma-fft}
here and need the following lemma instead.
\begin{lemma} \label{lemma-fft-decay} 
Suppose $\phi'(t)>0,\; t\in[0,1]$, $\phi(0)=0,\;\phi(1)=1$,
and
\begin{eqnarray}
  |\widehat{\phi'}(k)|, |\widehat{\psi'}(k)|\le C|k|^{-p},\;\forall|k|>M_{0}.\nonumber
\end{eqnarray}
Then for $n\le p-1$, we have
\begin{eqnarray}
\left|\int_{0}^{1}e^{i\psi}e^{-i2\pi\omega\phi}d\phi\right|\le\frac{P\left(\frac{\|\widehat{\phi'}\|_{1,M_0}+CM_{0}^{-p+1}}{\min\phi'},n\right)
}{|\omega|^{n}\left(\min\phi'\right)^{n}}M_{0}^{n}\sum_{j=1}^{n}(2\pi M_{0})^{-j}\left(\|\widehat{\psi'}\|_{1,M_0}+CM_{0}^{-p+1}\right)^{j}
\nonumber
\end{eqnarray}
 provided that $e^{i\psi}e^{-i2\pi\omega\phi}$ is a periodic function. 
Here 
$\|\widehat{\psi'}\|_{1,M_{0}}=\sum_{|k|\le M_{0}}|\widehat{\psi'}(k)|$ and $P(x,n)$ is
the same $(n-1)$th order polynomial as in Lemma \ref{lemma-fft}.
\end{lemma} 
\begin{proof} The proof is similar to the proof of Lemma
\ref{lemma-fft}. The only difference is that 
we need the following estimate instead of \myref{control-g}, 
\begin{eqnarray}
\max_{t}|\psi^{(n)}(t)| & \le & \sum_{k}|(2\pi k)^{n-1}\widehat{\psi'}(k)|\le(2\pi M_{0})^{n-1}\sum_{|k|\le M_{0}}|\widehat{\psi'}(k)|+(2\pi)^{n-1}C\sum_{|k|>M_{0}}|k|^{-p+n-1}\nonumber \\
 & \le & (2\pi M_{0})^{n-1}\left(\|\widehat{\psi'}\|_{1,M_{0}}+CM_{0}^{-p+1}\right).
\end{eqnarray}
 \end{proof}
Using this lemma and following an argument similar to that as
in the previous two theorems, we can prove the following theorem:
\begin{theorem} 
\label{theorem-decay-theta} Assume that the Fourier coefficients of the instantaneous
frequency $\theta'$,  the local mean $f_{0}$ and the envelope $f_{1}$
all have fast decay, i.e. there exists $C_{0}>0,\; p\ge 4$
such that 
\begin{eqnarray}
|\mathcal{F}(\theta')(k)|\le C_0 |k|^{-p},\quad |\mathcal{F}_{\theta}(f_{0})(k)|\le C_{0}|k|^{-p},\quad|\mathcal{F}_\theta(f_{1})(k)|\le C_{0}|k|^{-p} .
\end{eqnarray}
 If $L$ is large enough and the intial guess satisfies
\begin{eqnarray}
\|\mathcal{F}\left(\left(\theta^{0}-\theta\right)'\right)\|_{1}\le\pi M_0/2,
\end{eqnarray}
then, we have
\begin{eqnarray}
\|\mathcal{F}\left(\left(\theta^{m+1}-\theta\right)'\right)\|_{1}\le \Gamma_{0}(L/4)^{-p+2}+\frac{1}{2}C_0 M_0^{-p+1}
+\frac{1}{2}\left\Vert \mathcal{F}\left(\left(\theta^{m}-\theta\right)'\right)\right\Vert _{1},
\end{eqnarray}
 where $\Gamma_0>0$ is a constant determined by $C_0$, $M_0$ and $f_1$. 
\end{theorem}
\begin{remark}
 In the analysis presented in this section, we have assumed that the Fourier transform in the $\theta^m$-space, $\mathcal{F}_{\theta^m}(\cdot)$, is exact. In real computations, we need to first interpolate the signal from
a uniform grid in the physical space to a uniform grid in the
 $\theta^m$-space, then apply the Fast Fourier transform. This 
interpolation process would introduce some error. However, the 
interpolation error should be very small since we assume that the 
signal is well resolved by the sample points.
\end{remark}

\section{Periodic signal with sparse samples}
\label{section-sparse}

In this section, we will consider a more challenging case, the sample points $t_{j},\; j=1,\cdots,N$ are too few to resolve the signal. In this case, the algorithm presented in the last section does not 
apply directly. The reason is that the Fourier transform in the
$\theta^m$-space, $\mathcal{F}_{\theta^m}(\cdot)$, cannot be computed accurately by the interpolation-FFT method. One way to obtain the the Fourier transform in the $\theta^m$-space is to solve a linear system. Such method
is very expensive. Moreover, the resulting linear system is under-determined since we do not have sufficient number of sample points.

Thanks to the recent development of compressive sensing, we know that if the Fourier coefficients are sparse, then $l^1$ minimization would give an approximate solution from very few sample points.
Hence, we can use a $l^1$ minimization problem to generate the 
Fourier coefficients in the $\theta^m$-space in each step: 
\begin{itemize}
\item $\theta^{0}=\theta_{0},\; m=0$. 
\item Step 1: Solve the $l_{1}$ minimization problem to get the Fourier
transform of the signal $r^m$ in the $\theta^{m}$-coordinate: 
\begin{eqnarray}
\widehat{f}_{\theta^m}=\arg\min_{x\in \mathbb{R}^{N_b}} \|x\|_{1},\quad\mbox{subject to}\quad A_{\theta^{m}}\cdot x=f \label{opt-weight}
\end{eqnarray}
 where $A_{\theta^{m}}\in\mathbb{R}^{N_s\times N_b},\; N_s<N_b$, $N_s$ is
the number of samples and $N_b$ is the number of Fourier modes. 
$A_{\theta^{m}}(j,k)=e^{i2\pi k\ol{\theta}^{m}(t_{j})},\quad j=1,\cdots,N_s,\; k=-N_b/2+1,\cdots, N_b/2$
and $\ol{\theta}^{m}=\frac{\theta^{m}-\theta^{m}(0)}{\theta^{m}(T)-\theta^{m}(0)}$.
\item Step 2: Apply a cutoff function to the Fourier Transform of $r_{\theta_{k}^{n}}^{k-1}$
to compute $a^{m+1}$ and $b^{m+1}$: 
\begin{eqnarray}
a^{m+1} & = & \mathcal{F}_{\theta^{m}}^{-1}\left[\left(\widehat{f}_{\theta^{m}}\left(\omega+L_{\theta^{m}}\right)+\widehat{f}_{\theta^{m}}\left(\omega-L_{\theta^{m}}\right)\right)\cdot\chi\left(\omega/L_{\theta^{m}}\right)\right],\\
b^{m+1} & = & \mathcal{F}_{\theta^{m}}^{-1}\left[-i\cdot\left(\widehat{f}_{\theta^{m}}\left(\omega+L_{\theta^{m}}\right)-\widehat{f}_{\theta^{m}}\left(\omega-L_{\theta^{m}}\right)\right)\cdot\chi\left(\omega/L_{\theta^{m}}\right)\right],
\end{eqnarray}
where $\mathcal{F}_{\theta^m}^{-1}$ is the inverse Fourier transform defined in the $\theta^{m}$-coordinate: 
\begin{eqnarray}
\mathcal{F}_{\theta^m}^{-1}\left(\widehat{f}_{\theta^{m}}\right)(t_j)=\sum_{\omega=-N_b/2+1}^{N_b/2}\widehat{f}_{\theta^{m}}(\omega)e^{i2\pi\omega\ol{\theta}^m(t_j)},\quad j=1,\cdots,N_s,
\end{eqnarray}
 and $\chi$ is the cutoff function, 
\begin{eqnarray}
\chi(\omega)=\left\{ \begin{array}{cl}
1, & -1/2<\omega<1/2,\\
0, & \mbox{otherwise}.
\end{array}\right.
\end{eqnarray}
\item Step 3: Update $\theta^{m}$ in the $t$-coordinate:
\begin{eqnarray}
\dt'=P_{V_{M_{0}}}\left(\frac{d}{dt}\left(\arctan\left(\frac{b^{m+1}}{a^{m+1}}\right)\right)\right),\;\dt(t)=\int_{0}^{t}\dt'(s)ds,\quad\theta^{m+1}=\theta^{m}+\beta\dt,
\nonumber
\end{eqnarray}
 where $\beta\in[0,1]$ is chosen to make sure that $\theta^{m+1}$
is monotonically increasing: 
\begin{eqnarray}
\beta=\max\left\{ \alpha\in[0,1]:\frac{d}{dt}\left(\theta^{m}+\alpha\dt\right)\ge0\right\} ,
\end{eqnarray}
 and $P_{V_{M_{0}}}$ is the projection operator to the space 
$V_{M_{0}}=\mbox{span}\left\{ e^{i2k\pi t/T},k=-M_{0},\cdots,0,\cdots,M_{0}\right\}$ and $M_0$ is chosen 
{\it a priori}.

\item Step 4: If $\|\theta_{k}^{n+1}-\theta_{k}^{n}\|_{2}<\epsilon_{0}$,
stop. Otherwise, set $n=n+1$ and go to Step 1. 
\end{itemize}

Suppose the sample points $t_j,\;j=1,\cdots,N_s$ are selected at random
from a set of uniform 
grid $l/N_f,\; l=0,\cdots,N_f-1$, then the optimization problem \myref{opt-weight} in Step 1 can be rewritten in the following form: 
\begin{eqnarray}
\min\|x\|_{1},\quad\mbox{subject to}\quad \Phi_{\theta^{m}}\cdot x=\widetilde{f},
\label{opt-l2}
\end{eqnarray}
 where $\widetilde{f}=\sqrt{\frac{\left(\ot^{m}\right)'}{N_f}}\; f$ and $\Phi_{\theta^{m}}$ is obtained by selecting $N_s$ rows from an $N_f$ by $N_b$
matrix $U_{\theta^{m}}$ which is defined as $U_{\theta^{m}}(j,k)=\sqrt{\frac{\left(\ot^{m}\right)'}{N_f}}\cdot e^{i2\pi k\ot^m(t_{j})},\;j=1,\cdots,N_f,\; k=-N_b/2+1
,\cdots,N_b/2$. 
As we will show later, the columns of $U_{\theta^m}$
are approximately orthogonal to each other. This property
will play an important role in our convergence and stability analysis.

We remark that our problem is more challenging than the compressive sensing problem in the sense that we need not only to find the sparsest representation but also a basis parametrized by a phase function $\theta$ over 
which the signal has the sparsest representation. To overcome this
difficulty, we propose an iterative algorithm to solve this nonlinear
optimization problem.

\subsection{Exact recovery}
\label{section-sparse-exact}

\begin{theorem} \label{theorem-sparse-sample-exact} 
Under the same assumption as in Theorem \ref{theorem-sparse}, there exist $\eta_0>0,\; \eta_1>0$, such that
\begin{eqnarray}
\|\mathcal{F}\left(\left(\theta^{m+1}-\theta\right)'\right)\|_{1}\le \frac{1}{2}\left\Vert \mathcal{F}\left(\left(\theta^{m}-\theta\right)'\right)\right\Vert _{1},
\end{eqnarray}
provided that $L\ge \eta_0$ and $S\ge \eta_1$, where
$S$ be the largest number such that $\delta_{3S}(\Phi_{\theta^{m}})+3\delta_{4S}(\Phi_{\theta^{m}})<2$. Here
$\delta_S(A)$ is the $S$-restricted isometry constant of matrix $A$ given in \cite{CT06a}, which is the smallest number such that 
\begin{eqnarray}
  (1-\delta_S)\|c\|_{l^2}^2\le\|A_Tc\|_{l^2}^2\le (1+\delta_S)\|c\|_{l^2}^2,\nonumber
\end{eqnarray}
for all subsets $T$ with $|T|\le S$ and coefficients sequences $(c_j)_{j\in T}$.
 \end{theorem} 

To prove this theorem, we need to use the following theorem of Candes, Romberg, and Tao \cite{CRT06b}. 
\begin{theorem} \label{stable-l1} Let $S$
be such that $\delta_{3S}(A)+3\delta_{4S}(A)<2$, where $A\in\mathbb{R}^{n\times m},\; n<m$.
Suppose that $x_{0}$ is an arbitrary vector in $\mathbb{R}^{m}$
and let $x_{0,S}$ be the truncated vector corresponding to the $S$
largest values of $x_{0}$. Then the solution $x^{*}$ to the $l_{1}$
minimization problem 
\begin{eqnarray}
\min\|x\|_{1},\quad\quad\mbox{subject to}\quad Ax=f
\end{eqnarray}
satisfies 
\begin{eqnarray}
\|x^{*}-x_{0}\|_{1}\le C_{2,S}\cdot\|x_{0}-x_{0,S}\|_{1} .
\end{eqnarray}
 \end{theorem} 

Now we present the proof of Theorem \ref{theorem-sparse-sample-exact}.

\begin{proof}\textit{of Theorem \ref{theorem-sparse-sample-exact}}.
Using \myref{exp-error-a-ss} and \myref{exp-error-b-ss} in Appendix B, we have 
\begin{eqnarray}
|\Delta a^m|
& \le &2\sum_{\frac{L^m}{2}<k<\frac{3}{2}L^m}\left|\widehat{f}_{0,\theta^m}(k)\right|+\sum_{\frac{3}{2}L^m<k<\frac{5}{2}L^m}\left(\left|
\widehat{a}^m_{\theta^m}(k)\right|+\left|\widehat{b}^m_{\theta^m}(k)\right|\right)\nonumber\\
&&+\sum_{|k|>\frac{L^m}{2}}\left|\widehat{a}^m_{\theta^m}(k)\right|+2\sum_{\frac{L^m}{2}<k<\frac{3}{2}L^m}\left|\widehat{f}_{\theta^m}(k)-
\widehat{\widetilde{f}}_{\theta^m}(k)\right|\nonumber \\
 & \le & \Gamma_{0}Q(\al L)^{-n+1}\gamma +C_{2,S}\cdot\|\widehat{f}_{\theta^m}-\widehat{f}_{\theta^m,S}\|_{1},
\label{eqn-91}
\end{eqnarray}
where $\Gamma_0$ is a constant depending on $M_0, M_1, n$ and $\widehat{f}_{\theta^m,S}$ is the truncated vector corresponding to the $S$
largest values of $\widehat{f}_{\theta^m}$.

 Without loss of generality, we assume that $L^m>S/3$, and define $\widehat{\overline{f}}_{\theta^m,S}$ to be 
 \begin{eqnarray*}
   \widehat{\overline{f}}_{\theta^m,S}(k)=\left\{\begin{array}{ll}
\widehat{f}_{\theta^m}(k),& k\in [-L^m-S/6,-L^m+S/6]\cup [-S/6,S/6]\cup [L^m-S/6,L^m+S/6],\\
0,& \mbox{otherwise}.
\end{array}\right.
 \end{eqnarray*}
Then by the definition of $\widehat{f}_{\theta^m,S}$ and $\widehat{\overline{f}}_{\theta^m,S}$,
we have 
\begin{eqnarray}
\|\widehat{f}_{\theta^m}-\widehat{f}_{\theta^m,S}\|_{1} & \le & \|\widehat{f}_{\theta^m}-\widehat{\overline{f}}_{\theta^m,S}\|_{1}\nonumber\\
&=&\sum_{S/6<|k|<L^m-S/6}|\widehat{f}_{\theta^m}(k)|+\sum_{|k|>L^m+S/6}|\widehat{f}_{\theta^m}(k)|\nonumber \\
 & \le & \sum_{|k|>S/6}|\widehat{f}_{0,\theta^m}(k)|+\sum_{|k|>S/6}|\widehat{a}_{\theta^m}(k)|+\sum_{|k|>S/6}|\widehat{b}_{\theta^m}(k)|\nonumber \\
 & \le & \Gamma_{1}QS^{-n+1}\gamma.\quad\quad
\label{eqn-92}
\end{eqnarray}
Substituting \myref{eqn-92} into \myref{eqn-91}, we get
\begin{eqnarray}
|\Delta a^m| \le  \left(\Gamma_0(\al L)^{-n+1}+C_{2,S}\Gamma_1S^{-n+1}\right)Q\gamma.
\end{eqnarray}
Similarly, we obtain
\begin{eqnarray}
|\Delta b^m| \le  \left(\Gamma_0(\al L)^{-n+1}+C_{2,S}\Gamma_1S^{-n+1}\right)Q\gamma.
\end{eqnarray}
Using these two key estimates and follow the same argument as that in
the proof of Theorem \ref{theorem-sparse}, we can complete the proof
of Theorem \ref{theorem-sparse-sample-exact}.
\end{proof}
\begin{remark}
The above result on the exact recovery of signals with sparse samples 
can be generalized to the case that we consider in Theorem 
\ref{theorem-decay} by combining the argument of the above theorem 
with the idea presented in the proof of Theorem \ref{theorem-decay}. In 
this case, we can recover the signal with an error which is 
determined by $L$, $S$ and the decay rates of 
$\widehat{f}_{0,\theta},\widehat{f}_{1,\theta}$ and $\widehat{\theta'}$.
\end{remark}

In Theorem \ref{theorem-sparse-sample-exact}, we assume that in each step, the condition $\delta_{3S}(\Phi_{\theta^{m}})+3\delta_{4S}(\Phi_{\theta^{m}})<2$
is satisfied. Using the definition of $\delta_S$, it is easy to see that $\delta_{3S}\le \delta_{4S}$. Thus, a sufficient condition to satisfy
$\delta_{3S}(\Phi_{\theta^{m}})+3\delta_{4S}(\Phi_{\theta^{m}})<2$ is 
to require $\delta_{4S}(\Phi_{\theta^{m}})<1/2$.

In compressive sensing, there is a well-known result by Candes and Tao 
in \cite{CT06b}. This result states that if the matrix 
$\Phi\in \mathbb{R}^{M\times N}$ is obtained by selecting $M$ rows 
at random
from an $N\times N$ Fourier matrix $U$ where $U_{j,k}=\frac{1}{\sqrt{N}}e^{i2\pi jk/N},\; j,k=1,\cdots N$, then the condition 
$\delta_{S}(\Phi)<1/2$ is satisfied with an 
overwhelming probability provided that
\begin{eqnarray}
  S\le C\frac{M}{(\log N)^6},
\end{eqnarray}
where $C$ is a constant.

In our formulation (see \myref{opt-l2}), the matrix $\Phi_{\theta^{m}}$ also consists of $N_s$ rows of a $N_f$-by-$N_b$ matrix $U_{\theta^m}$. 
The main difference is that the matrix $U_{\theta^m}$ is not a standard Fourier matrix. Instead it is a Fourier matrix in the $\theta^m$-space which makes it non-orthonormal. As a result, we cannot apply the result of Candes and Tao in \cite{CT06b} directly. Fortunately,
we have the following result by slightly modifying the arguments used in \cite{CT06b} which can be applied to matrix $U_{\theta^m}$.
\begin{theorem} 
\label{theorem-est-S-0}
If $\nu_0=\max_{k,j}|U_{\theta}^*U_{\theta}-I)_{k,j}|\le \frac{1}{16N_b}$, where $U_{\theta}^*$ is the conjugate transpose of $U_\theta$, 
the condition $\delta_{S}(\Phi_{\theta})<1/2$
holds with probability $1-\delta$ provided that
\begin{eqnarray}
N_s\ge C\cdot \max(\ol{\theta})'\left(S\log^2 N_b-\log \delta\right)\log ^4 N_b,
\end{eqnarray}
where $N_s$ is the number of the samples, $N_b$ is the number of elements
in the basis.
\end{theorem}
This theorem shows that if the columns of $U_{\theta^m}$ are approximately orthogonal to each other, it has a property  similar to the standard Fourier matrix.
Consequently, we need only to estimate the mutual coherence of the columns of the matrix $U_{\theta^m}$ for $\theta^m\in V_{M_0}$. 
\begin{lemma} \label{lemma-integral-dtheta} 
Let $\phi'(t)\in V_{M_{0}},\; t\in[0,1]$
and $\phi(0)=0,\;\phi(1)=1,\;\phi'>0$, $t_{j}=j/L,\; j=0,\cdot,L-1$
is a uniform grid over $[0,1]$, then for any $n\in\mathbb{N}$, there
exists $C(n)>0$, such that 
\begin{eqnarray}
\frac{1}{L}\sum_{j=0}^{L-1}\phi'(t_{j})e^{i2\pi k\phi(t_{j})}\le C(n)\max\left\{ \left(\frac{k\|\widehat{\phi'}\|_{1}}{L}\right)^{n},\left(\frac{2M_{0}}{L}\right)^{n}\right\} .
\end{eqnarray}
 \end{lemma} 
The proof of this lemma is deferred to Appendix C.

Using this lemma, we can show that the condition $\nu_0=\max_{k,j}|U_{\theta^m}^*U_{\theta^m}-I)_{k,j}|\le \frac{1}{16N_b}$ is satisfied as long as $N_f\ge C \|\mathcal{F}((\ot^m)')\|_{1} N_b$ where $C$ is a constant determined by $N_b$. This leads to the following theorem.
\begin{theorem} 
\label{theorem-est-S}
Suppose the sample points $t_j,\;j=1,\cdots,N_s$ are selected at random
 from a set of uniform grid $l/N_f,\; l=0,\cdots,N_f-1$. If
\begin{eqnarray*}
  N_f\ge C \|\mathcal{F}((\ot^m)')\|_{1} N_b
\end{eqnarray*}
in $(m+1)$st step, we have $\delta_{S}(\Phi_{\theta^m})<1/2$
holds with probability $1-\delta$ provided that
\begin{eqnarray}
N_s\ge C\cdot \max[(\ol{\theta}^{m})']\left(S\log^2 N_b-\log \delta\right)\log ^4 N_b,
\end{eqnarray}
 where $N_s$ is the number of the samples, $N_b$ is the number of elements in the basis.
\end{theorem}

The above result shows that if the sample points are selected at random,
in each step, with probability $1-\delta$, we can get the right answer. 
This does not mean that the 
whole iteration converges to the right solution with an overwhelming probability. If the iteration is run up to the $n$th step, the probability 
that all these $n$ steps are successful is $1-n\delta$. If $n$ is large, 
the probability could be small even if $\delta$ is very small.

\subsection{Uniform estimate of $\delta_{S}(\Phi_{\theta^{m}})$ during the iteration}
 
In order to make sure that the iterative algorithm would converge with a high probability, we have to obtain an uniform estimate of $\delta_{S}(\Phi_{\theta^{m}})$ during the iteration. More precisely, we need to prove that 
with an overwhelming probability, 
\begin{eqnarray}
\sup_{\theta\in W_{M_0}}\delta_{S}(\Phi_{\theta})\le 1/2,
\end{eqnarray}
 where 
$W_{M_0}=\{\phi\in C^\infty[0,1]: \phi(0)=0, \phi(1)=1,\phi'\in V_{M_0},  \;\phi'(t)>0,\; \forall t\in [0,1]\}$.

The analysis below shows that this is true even if the number of sample points is in the same order as that required by Theorem \ref{theorem-est-S}.
There are two key observations in this analysis. The first one is that the difference between $\delta_S(\Phi_{\ol{\theta}})$ and $\delta_S(\Phi_{\ol{\phi}})$
would be small if $\ot, \ol{\phi}\in W_{M_0}$ and $\|\ot-\ol{\phi}\|_\infty$ is small. Actually, we can make $|\delta_S(\Phi_{\ol{\theta}})-\delta_S(\Phi_{\ol{\phi}})|\le \frac{1}{4}$ 
as long as $\|\ot'-\ol{\phi}'\|_\infty\le r = O(N_b^{-5/2}M_0^{-1})$. The second observation is that $W_{M_0}$ is bounded and finite dimensional which implies compactness. Then for any $r>0$, 
there exist a finite subset $A_r\subset W_{M_0}$, such that for any $\ot\in W_{M_0}$, there exists $\ol{\phi}_j\in A_r$, such that $\|\ot'-\ol{\phi}'_j\|_\infty\le r$. 

Based on these two observations, we can show that
\begin{eqnarray}
  \sup_{\theta\in W_{M_0}}\delta_{S}(\Phi_{\theta})\le \sup_{\phi\in A_r}\delta_{S}(\Phi_{\phi})+1/4.
\end{eqnarray}
Then by the union bound, we have
\begin{eqnarray}
 P\left( \sup_{\theta\in W_{M_0}}\delta_{S}(\Phi_{\theta})>1/2\right)\le P\left(\sup_{\phi\in A_r}\delta_{S}(\Phi_{\phi})>1/4\right) \le
|A_r|\sup_{\phi\in A_r}P\left(\delta_{S}(\Phi_{\phi})>1/4\right). 
\quad \quad
\end{eqnarray}
It is sufficient to prove that 
\begin{eqnarray}
P\left(\delta_{S}(\Phi_{\phi})>1/4\right)\le \delta/|A_r|,\quad \forall \phi\in A_r\subset W_{M_0},
\end{eqnarray}
which is true as long as 
\begin{eqnarray}
\label{est-s-uniform-0}
N_s\ge C\cdot \max_{\theta\in A_r}\|\theta'\|_\infty \left(S\log^2 N_b+\log |A_r|-\log \delta\right)\log ^4 N_b.
\end{eqnarray}
Now, we need only to choose a proper $r$ and estimate the corresponding $|A_r|$.
\begin{lemma}
\label{lemma-cover-wb}
Let $W=\{\phi\in C^\infty[0,1]: \phi(0)=0, \phi(1)=1,\phi'\in V_{M_0},  \;\phi'(t)>0,\; \forall t\in [0,1]\}$. For any $r>0$, one can find a finite subset $A_r$ of $W$ with 
cardinality 
\begin{eqnarray}
\label{cover-number}
  |A_r|\le \left(\frac{16\pi M_0^2}{r}+1\right)^{2M_0},
\end{eqnarray}
such that for all $\psi\in W$, there exists $\phi\in A_r$ such that $\|\psi'-\phi'\|_\infty\le r$ and $\|\psi-\phi\|_\infty\le r$.
\end{lemma}
\begin{proof}
Let $\ol{W}=\{\phi':\phi\in W\}$. Then for all $\overline{\psi}\in \ol{W}$, we have the following Fourier representation 
\begin{eqnarray}
\label{coe-dw}
  \ol{\psi}(t)=1+\sum_{j=1}^{M_0}(c_j\cos(2\pi j t)+d_j\sin(2\pi j t))>0,\quad \forall t\in [0,1].
\end{eqnarray}
Since $\int_0^t \psi(s)ds\in W$ according to the definition of $\ol{W}$, then $\int_0^1 \psi(s)ds=1$, so the constant in the above Fourier
representation is $1$.

By multiplying $1+\cos(2\pi jt)$ to both sides of \myref{coe-dw} and integrating over $[0,1]$ with respect to $t$, we get 
\begin{eqnarray}
  1+c_j/2\ge 0,\nonumber
\end{eqnarray}
which implies that $c_j\ge -2$, where we have used the fact that $1+\cos(2\pi j t)\ge 0$.

On the other hand, multiplying $-1+\cos(2\pi jt)$ to both sides of \myref{coe-dw} and taking integral over $[0,1]$ with respect to $t$, we have $c_j\le 2$.
Combining these two results, we have 
\begin{eqnarray}
  |c_j|\le 2.
\end{eqnarray}

Similarly, by multiplying $\sin(2\pi jt)\pm 1$ to both sides of \myref{coe-dw} and taking integral over $[0,1]$ with respect to $t$, we obtain
\begin{eqnarray}
  |d_j|\le 2.
\end{eqnarray}

Now, we have proven that for any function in $\ol{W}$, its Fourier coefficients are bounded by $2$.

Let $h=r/(2 M_0),\; L_r=\lceil 4/h\rceil$, $Z_r=\{-2, -2+h, -2+2h,\cdots, -2+(L_r-1) h\}$.

For any $\ol{\psi}\in \ol{W}$, we know that its Fourier coefficients $c_j,d_j\in [-2,2], \; j=1,\cdots, M_0$, then one can find $a_j, b_j\in Z_r$ correspondingly such that 
\begin{eqnarray}
|a_j-c_j|&\le& h/2=r/(4 M_0),\quad j=1,\cdots,M_0,\nonumber\\
|b_j-d_j|&\le& h/2=r/(4 M_0),\quad j=1,\cdots,M_0,\nonumber
\end{eqnarray}
which implies that there exists $y\in \ol{Y}_r$ such that 
\begin{eqnarray}
  \|\psi-y\|_\infty\le \sum_{j=1}^{M_0}(|a_j-c_j|+|b_j-d_j|)\le 2\pi M_0^2 h = r/2
\end{eqnarray}
where $\ol{Y}_r$ is defined as follows
\begin{eqnarray}
\ol{Y}_r=\{y=\sum_{j=1}^{M_0}(a_j\cos(2\pi j t)+b_j\sin(2\pi j t)): a_j,b_j\in Z_r, \; B_{r/2}(y)\cap \ol{W}\ne \emptyset\},\nonumber 
\end{eqnarray}
and $B_{r/2}(y)=\{z\in V_{M_0}:\|z-y\|_\infty\le r/2\}$.

By the definition of $\ol{Y}_r$, one can get 
\begin{eqnarray}
  |\ol{Y}_r|\le |Z_r|^{2M_0}= L_r^{2M_0}\le \left(\frac{8 M_0}{r}+1\right)^{2M_0}.
\end{eqnarray}

Suppose $\ol{Y}_r=\{y_1,y_2,\cdots, y_{|\ol{Y}_r|}\}$, by the definition of $\ol{Y}_r$, for each $y_j$, there exists $\ol{\phi}_j\in \ol{W}$ such that $\ol{\phi}_j\in B_{r/2}(y)$. 
We can get a finite subset $\ol{A}_r$ of $\ol{W}$ by collecting all these $\ol{\phi}_j$ together and obviously $|\ol{A}_r|=|\ol{Y}_r|$.

Finally, let 
\begin{eqnarray}
A_r=\left\{\int_0^t \ol{\phi}(s)ds:\quad \ol{\phi}\in \ol{A}_r\right\}.
\end{eqnarray}
Then, for any $\psi\in W$, there exists $\phi_j\in A_r$ and $y_j\in \ol{Y}_r$, such that 
\begin{eqnarray}
  \|\psi'-\phi'_j\|_\infty \le \|\psi'-y_j\|_\infty+\|y_j-\phi'_j\|_\infty\le r/2+r/2=r.
\end{eqnarray}
Moreover, we have
\begin{eqnarray}
  \|\psi-\phi_j\|_\infty \le \int_0^1|\psi'(s)-\phi'_j(s)|ds \le r,
\end{eqnarray}
where we have used the fact that $\psi(0)=\phi_j(0)=0$ to eliminate the integral constant.
\end{proof}
\begin{remark}
  By multiplying $c_j\cos(2\pi jt)+d_j\sin(2\pi jt)\pm \sqrt{c_j^2+d_j^2}$ to both sides of \myref{coe-dw} and 
taking integral over $[0,1]$ with respect to $t$, we have
\begin{eqnarray}
\label{coe-bound}
c_j^2+d_j^2\le 4, \quad j=1,\cdots,M_0.
\end{eqnarray}
This implies a sharper estimate of $|A_r|$, 
\begin{eqnarray}
  |A_r|\le \left(\frac{8\pi M^2_0}{r^2}\right)^{M_0} .
\end{eqnarray}
Also, \myref{coe-bound} gives us a bound for $\|\phi'\|_\infty$ in $W_{M_0}$,
\begin{eqnarray}
  \label{bound-dtheta}
\sup_{\phi\in W_{M_0}}\|\phi'\|_\infty\le 4M_0+1.
\end{eqnarray}
which will be used later.
\end{remark}

It remains to choose a proper $r$. First, we show that 
the difference of $\delta_S$ between two matrices can be controlled by 
the difference of each element.
\begin{proposition}
\label{delta-s-diff}
  Let $A, B$ are two $M$ by $N$ matrices, $M<N$ and the columns of $A$ are normalized to be unit vectors in $l^2$ norm.
Then, for any $S\in \mathbb{N}$, we have
\begin{eqnarray}
  |\delta_S(A)-\delta_S(B)|\le (2\e \sqrt{M}+\e^2M)S,
\end{eqnarray}
where $\e=\max_{i,j}|A_{ij}-B_{ij}|$. 
\end{proposition}
\begin{proof}
  By the definition of $\delta_S$, we need only to prove that for all subsets $T$ with $|T|\le S$ and coefficients sequences $(c_j)_{j\in T}$,
  \begin{eqnarray}
    \left|\|A_Tc\|_2^2-\|B_Tc\|_2^2\right|\le (2\e \sqrt{M}+\e^2M)S \|c\|_2^2.
  \end{eqnarray}
This can be verified by a direct calculation:
\begin{eqnarray}
  \left|\|A_Tc\|_2^2-\|B_Tc\|_2^2\right|&=&|\sum_{i,j\in T}c_ic_j(A_i^TA_j-B_i^TB_j) |\nonumber\\
&= &|\sum_{i,j\in T}c_ic_j(D_i^TA_j+A_i^T D_j+D_i^T D_j) |\nonumber\\
&\le &\max_{i,j\in T} |D_i^TA_j+A_i^T D_j+D_i^T D_j| \sum_{i,j\in T}|c_ic_j| \nonumber\\
&\le & |T| \|c\|_2^2\max_{i,j\in T} (\|D_i\|_2\|A_j\|_2+\|A_i\|_2 \|D_j\|_2+\|D_i\|_2\|D_j\|_2)  \nonumber\\
&\le & (2\e \sqrt{M}\max_{i\in \mathbb{Z}_N} \|A_i\|_2+\e^2M)S \|c\|_2^2.
\end{eqnarray}
In the above derivation, $D=B-A$, $A_i, A_j$ are $i$th and $j$th columns of $A$.
\end{proof}
Using the above proposition, we obtain the following result:
\begin{corollary}
\label{coro-diff-delta}
  Let $\ot,\; \ol{\phi} \in W$, then 
  \begin{eqnarray}
    |\delta_S(\Phi_{\ol{\theta}})-\delta_S(\Phi_{\ol{\phi}})|\le \frac{1}{8}   ,
  \end{eqnarray}
provided that $|\ol{\theta}'-\ol{\phi}'|\le C N_b^{-2}M_0^{-1/2}$, where $C$ is an absolute constant.
\end{corollary}
\begin{proof} We need only to show that the difference between $\Phi_{\ol{\theta}}$ and $\Phi_{\ol{\phi}}$ can be controlled by 
$|\ol{\theta}'-\ol{\phi}'|$. This is quite straightforward using the definition of $\Phi_{\ol{\theta}}$ and $\Phi_{\ol{\phi}}$:
  \begin{eqnarray}
    |\Phi_{\ol{\theta}}(j,k)-\Phi_{\ol{\phi}}(j,k)|&=&\frac{1}{\sqrt{N_s}}\left|\sqrt{\ol{\theta}'(t_j)}e^{i2\pi k \ot(t_j)}-\sqrt{\ol{\phi}'(t_j)}e^{i2\pi k \ol{\phi}(t_j)}\right|\nonumber\\
&\le &\frac{|\sqrt{\ol{\theta}'(t_j)}-\sqrt{\ol{\phi}'(t_j)}|}{\sqrt{N_s}}+\frac{\sqrt{\ol{\theta}'(t_j)}}{\sqrt{N_s}}\left|e^{i2\pi k (\ot(t_j)- \ol{\phi}(t_j)}-1\right|\nonumber\\
&\le &\frac{|\sqrt{\ol{\theta}'(t_j)}-\sqrt{\ol{\phi}'(t_j)}|}{\sqrt{N_s}}+\frac{\sqrt{\ol{\theta}'(t_j)}}{\sqrt{N_s}}2\pi k |\ot(t_j)- \ol{\phi}(t_j)|\nonumber\\
&\le & \frac{\sqrt{\e}}{\sqrt{N_s}}+\frac{2\pi N_b \e \sqrt{4M_0+1}}{\sqrt{N_s}},
  \end{eqnarray}
where we have used the estimate $\|\ol{\theta}'\|_\infty\le 4M_0+1$ given in \myref{bound-dtheta}. 
Using Proposition \ref{delta-s-diff} and the fact that $S\le N_b$, we can complete the proof.
\end{proof}

Combining Lemma \ref{lemma-cover-wb}, Corollary \ref{coro-diff-delta} and \myref{est-s-uniform-0}, we have the following theorem,
\begin{theorem} 
\label{theorem-est-S-uniform}
$\sup_{\theta\in W_{M_0}}\delta_{S}(\Phi_{\theta})\le 1/2$
holds with probability $1-\delta$ provided that
\begin{eqnarray}
N_s\ge C\cdot (4M_0+1)\left(S\log^2 N_b+M_0\log N_b-\log \delta\right)\log ^4 N_b,
\end{eqnarray}
 where $N_s$ is the number of the samples, $N_b$ is the number of elements
in the basis.
\end{theorem}
\begin{remark}
Comparing with the condition stated in Theorem \ref{theorem-est-S}, we 
require extra $M_0\log^5 N_b$ samples in order to get the uniform estimate.
But this number $M_0\log^5 N_b$ can be absorbed by $S\log^6 N_b$, since $S$ is larger than $M_0$. Thus the condition to get an uniform estimate is essentially the same as that in Theorem \ref{theorem-est-S}.
\end{remark}

\section{Numerical results}
In this section, we will perform several numerical experiments to 
confirm our theoretical results presented in the previous section 
and to demonstrate the performance of the algorithm based on 
the weighted $l^1$ optimization.

\vspace{3mm}
\noindent
\textbf{Example 1: Exact recovery for a well-resolved signal} 

\noindent
The first example is a well-resolved periodic signal. In this example,
the mean and the envelope have a sparse Fourier representation in 
the $\theta$-space and the instantaneous frequency has a sparse Fourier 
spectrum in the physical space. The signal we use is generated by 
the following formula: 
\begin{eqnarray}
  \label{exact-recover}
&&\theta=20\pi t +2\cos2\pi t +2\sin4\pi t,\quad \ot=\theta/10\nonumber\\
&&a_0=2+\cos\ot+2\sin2\ot+\cos3\ot,\quad a_1=3+\cos\ot+\sin3\ot\nonumber\\
&&f=a_0+a_1\cos\theta.
\end{eqnarray}
This signal is sampled over a uniform mesh of 256 points such that 
there are about 12 samples in each period of the signal on average.

\begin{figure}

    \begin{center}

\includegraphics[width=0.45\textwidth]{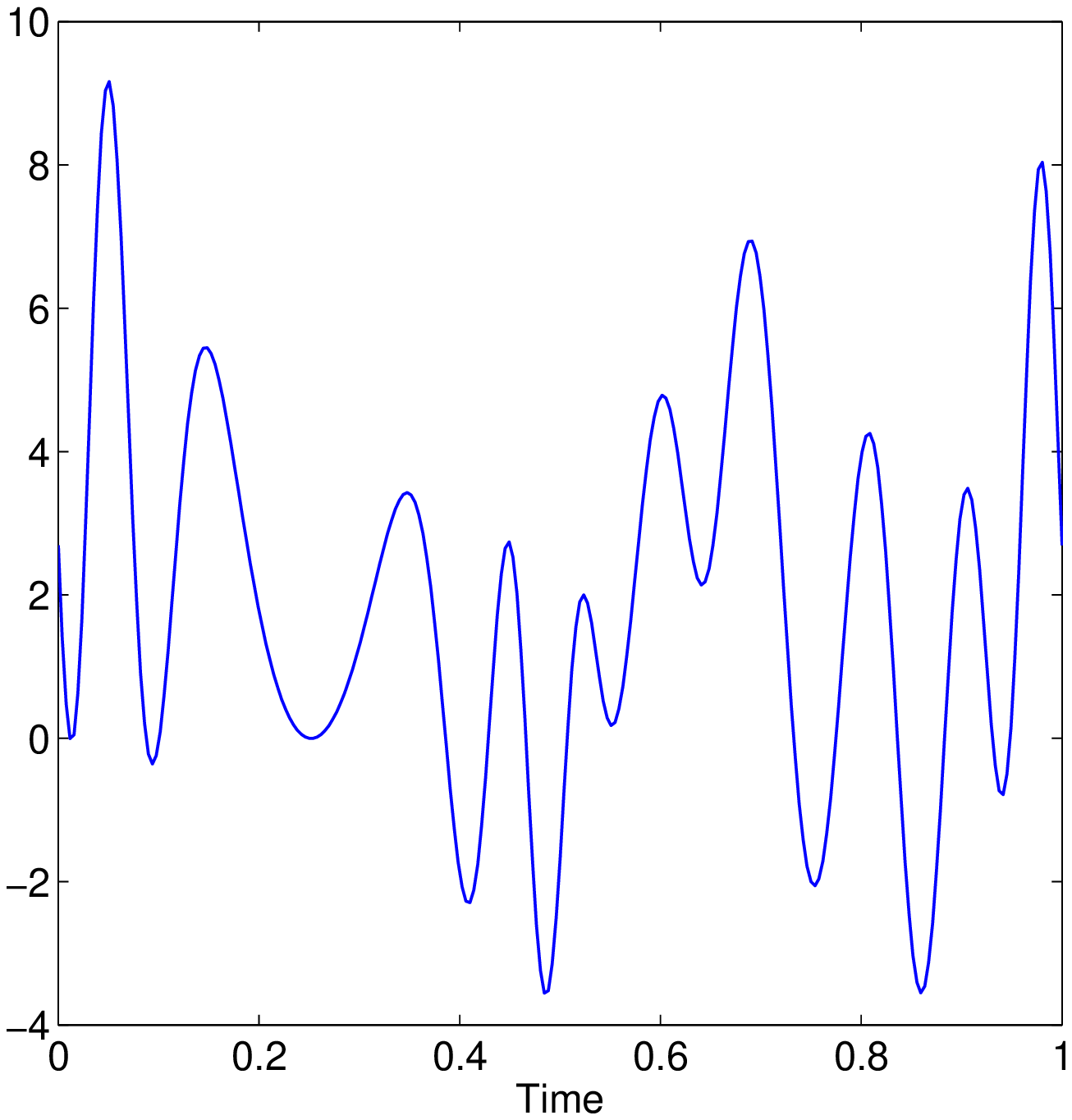}
\includegraphics[width=0.45\textwidth]{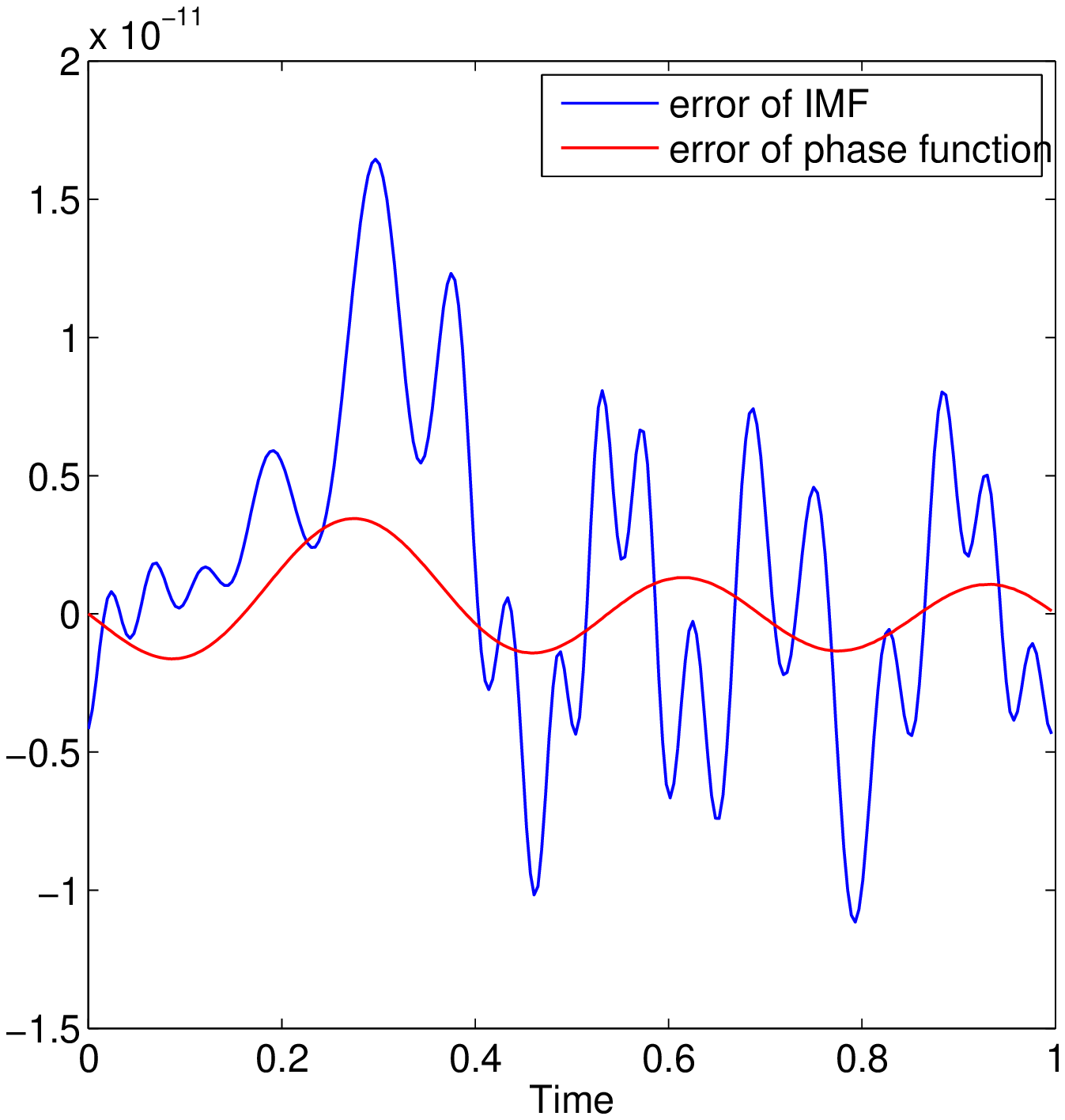}
     \end{center}
    \caption{  \label{error-direct}Left: Original signal; Right: Error of the IMF and the phase function.}
\end{figure}
The numerical results are shown in Fig. \ref{error-direct} and Fig. \ref{error-fft}.
In Fig. \ref{error-direct}, we can see that our algorithm indeed 
recovers the exact decomposition of this signal.
This is also consistent with the theoretical result we obtained in Theorem \ref{theorem-sparse}. The result shown in Fig. \ref{error-direct}
is obtained by applying the non-uniform Fourier transform directly. 
As we proposed in our algorithm, for a well-resolved signal, it is
more efficient to use a combination of interpolation and FFT. This 
procedure would introduce some interpolation error, 
however the computation is accelerated tremendously. 
\begin{figure}
    \begin{center}
\includegraphics[width=0.45\textwidth]{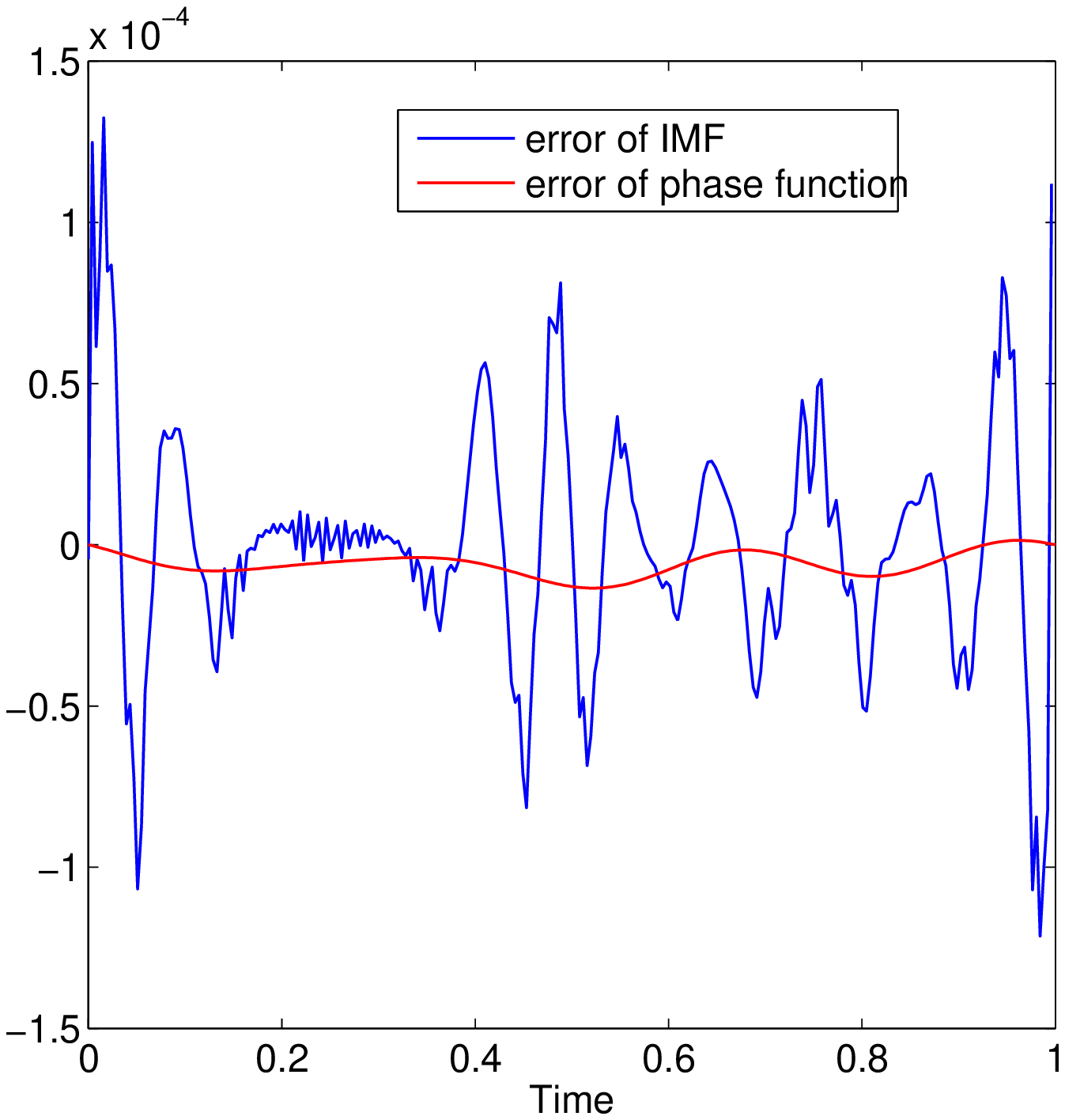}
\includegraphics[width=0.45\textwidth]{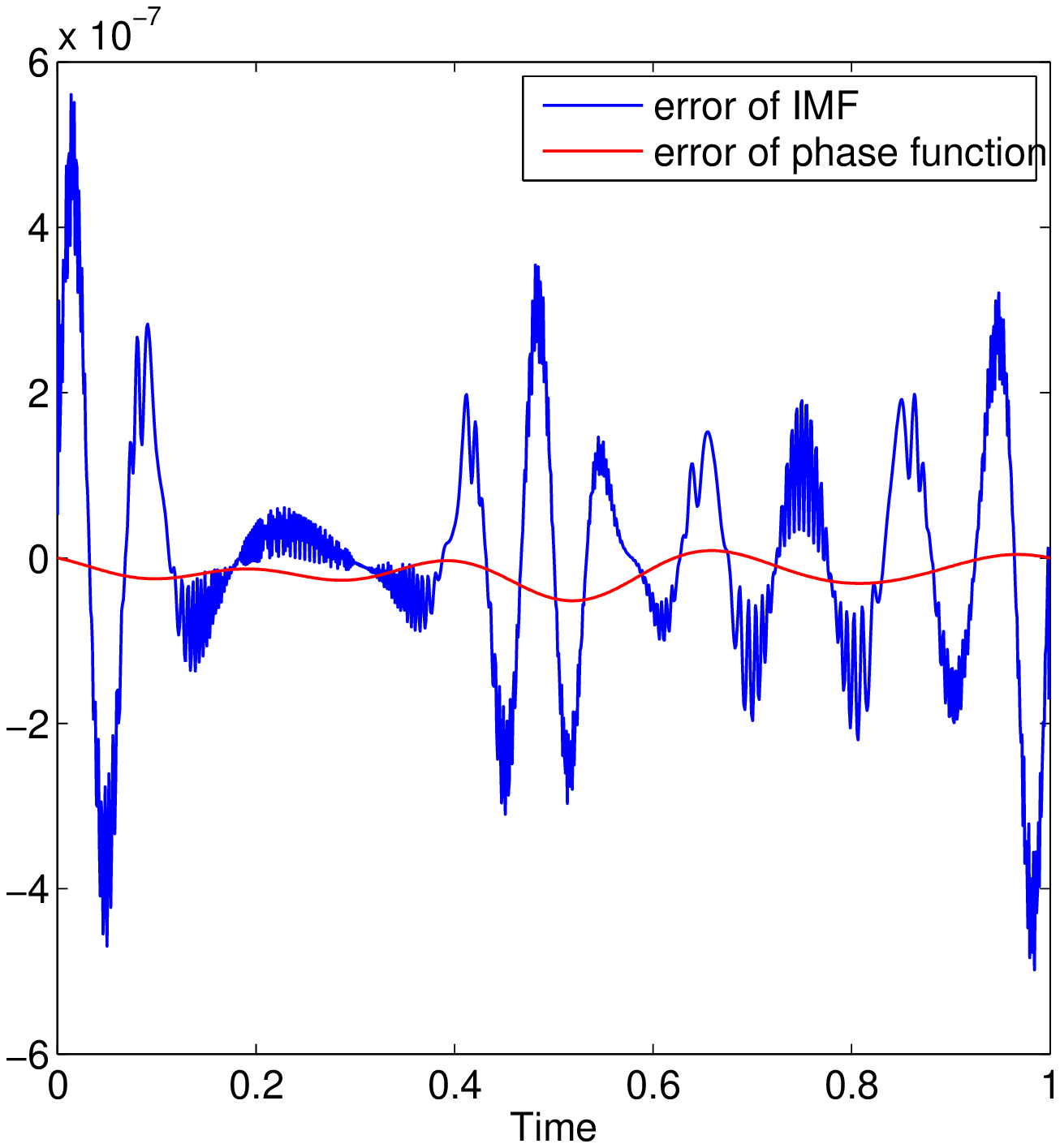}
     \end{center}
    \caption{  \label{error-fft}Left: Error of the IMF and the phase function with 256 uniform samples; Right: 
Error of the IMF and the phase function with 1024 uniform samples.}
\end{figure}
As we see in Fig. \ref{error-fft}, if we use the FFT-based algorithm, the error increase to the order of $10^{-4}$ instead of $10^{-11}$ in the previous result when we used the non-uniform Fourier transform. If we increase 
the number of sample points to 1024, the order of error decreases to 
$10^{-7}$. This indicates that the main source of error comes from the 
interpolation error.

In our previous paper \cite{HS12}, we have shown many numerical results to
demonstrate the stability of our algorithm. These numerical examples 
confirm the theoretical results presented in Theorem \ref{theorem-decay} and Theorem \ref{theorem-decay-theta}. We will not reproduce these numerical
examples in this paper.

\vspace{3mm}
\noindent
\textbf{Example 2: Exact recovery for a signal with random samples}

\noindent
The second example is designed to confirm the result of Theorem \ref{theorem-sparse-sample-exact}. This example shows that 
for a signal with a sparse structure, our algorithm is capable of producing the exact decomposition even if it is poorly sampled. The signal is given
below in \myref{exact-recover-sparse}.  
\begin{eqnarray}
  \label{exact-recover-sparse}
&&\theta=200\pi t -10\cos2\pi t -2\sin4\pi t,\quad \ot=\theta/(100)\nonumber\\
&&a_0=\cos\ot,\quad a_1=3+\cos\ot+\sin2\ot\nonumber\\
&&f=a_0+a_1\cos\theta.
\end{eqnarray}
The number of sample points is set to be 120. These sample points are 
selected at random over 4096 uniformly distributed points. 
On average, there are only 1.2 points in each period of the signal. 
We test 100 independent samples and our algorithm is able to recover
the signal for 97 samples, which gives $97\%$ success rate.  
Fig. \ref{error-sparse-exact} gives one of the successful samples.
\begin{figure}

    \begin{center}

\includegraphics[width=0.43\textwidth]{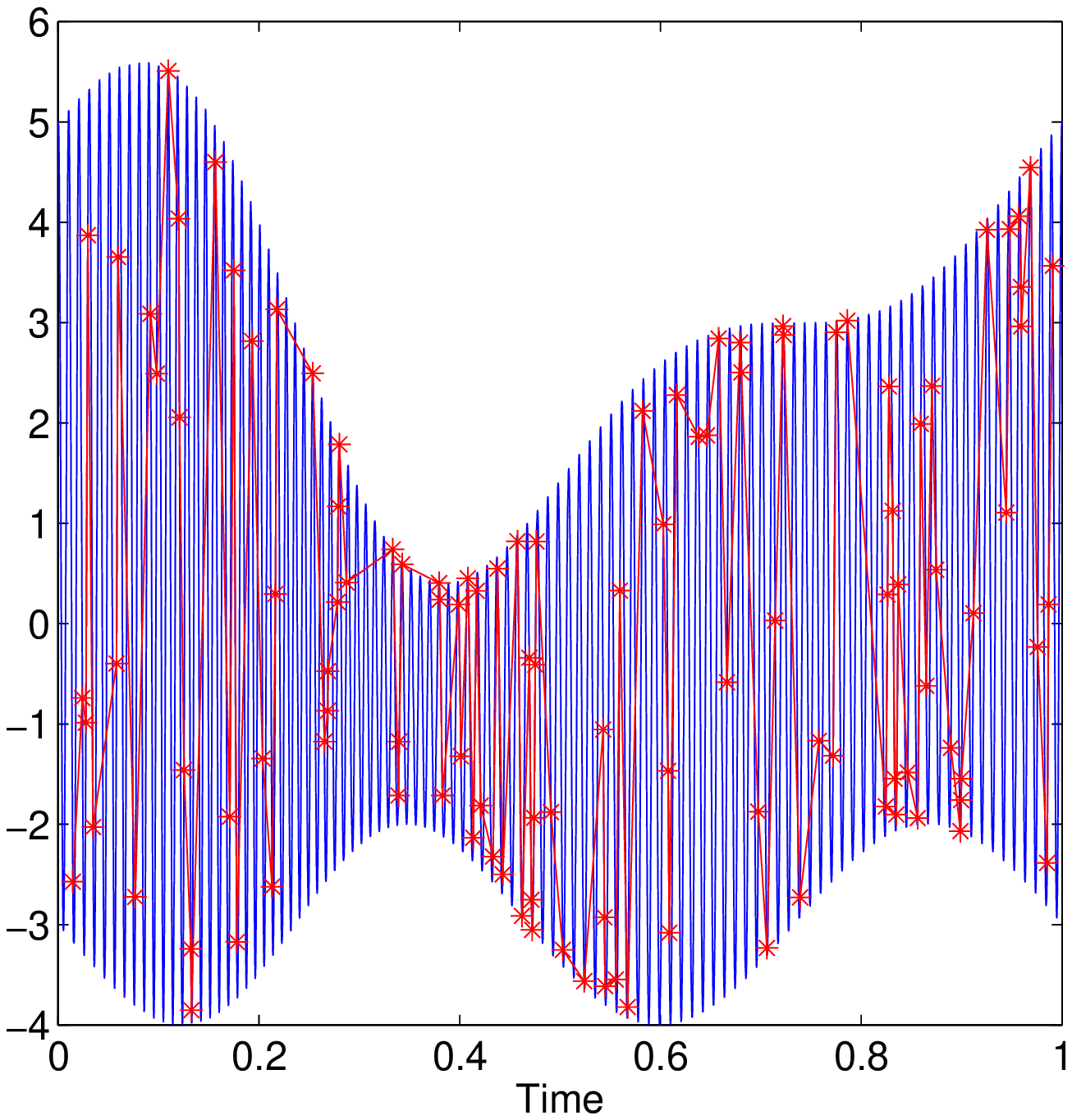}
\includegraphics[width=0.47\textwidth]{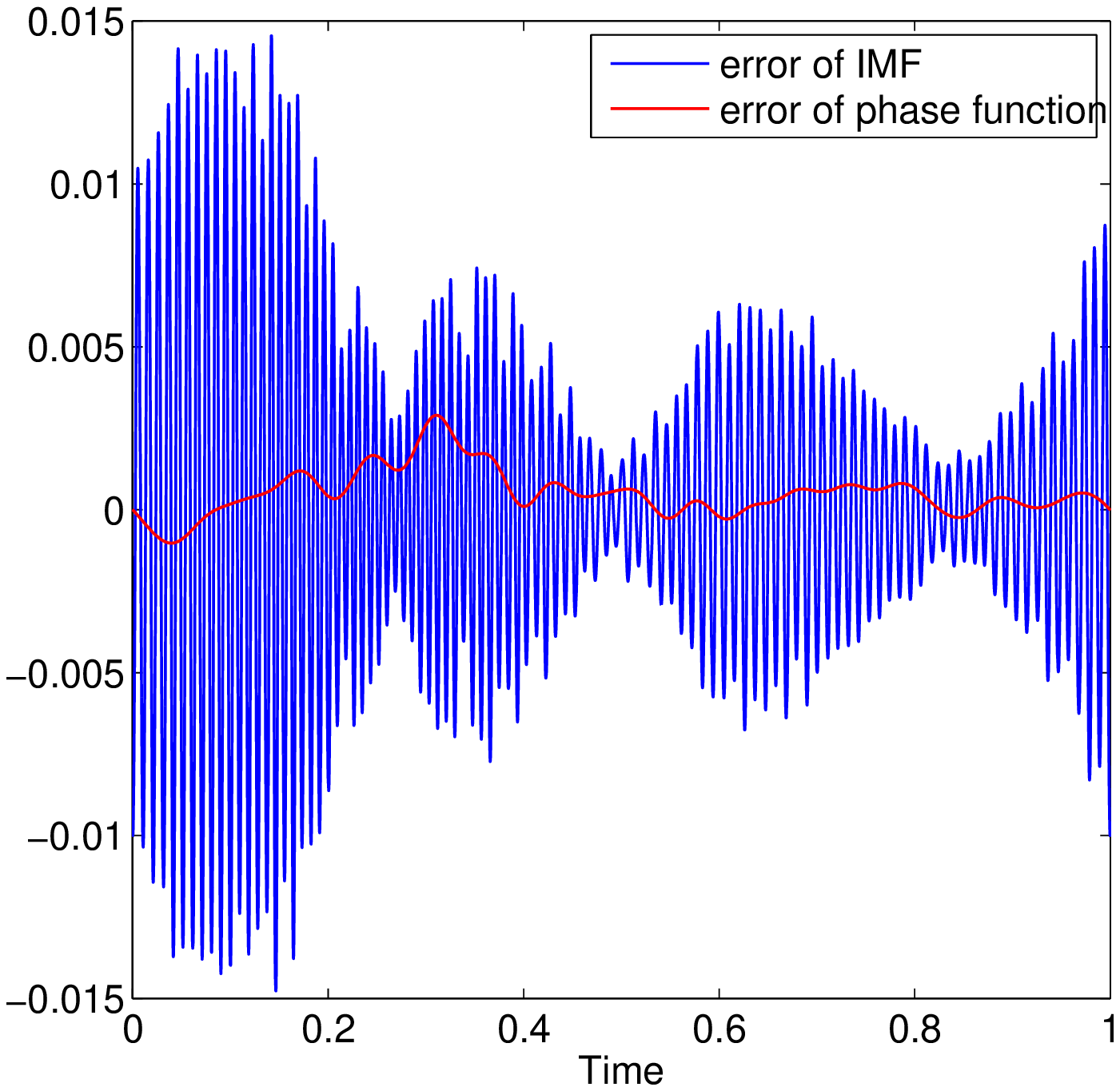}
     \end{center}
    \caption{  \label{error-sparse-exact}Left: Original signal and the sample points; Right: Error of the IMF and phase function.}
\end{figure}

The right panel of Fig. \ref{error-sparse-exact} shows that the order of error is $10^{-2}$ for IMF and $10^{-3}$ for the phase function.
In the computation, the $l^1$ optimization problem is solved approximately in each step of the iteration. This is the reason that the 
error is much larger than the round-off error of the computer. If we increase the accuracy in solving the $l^1$ optimization problem, the algorithm
would give a more accurate result. However the computational cost also 
increases as a consequence. We also reduce the number of sample points to 
80 and carry out the same test for 100 times. In this case, the recovery 
rate was 46 out of 100.

\vspace{3mm}
\noindent
\textbf{Example 3: Approximate recovery for a signal with random samples}

\noindent
In this example, we will check the stability of our algorithm for a sparsely sampled signal. 
The signal is generated by \myref{sparse-stable},
\begin{eqnarray}
  \label{sparse-stable}
&&\theta=\theta_0 + 0.1\sin(120\pi t),\nonumber\\
&&a_0=\cos(2\pi t),\quad a_1=3+\cos(2\pi t)+\sin(4\pi t)\nonumber\\
&&f=a_0+a_1\cos\theta+0.1X(t).
\end{eqnarray}
where $\theta_0$ is the $\theta$ given in \myref{exact-recover-sparse}, and $X(t)$ is the Gaussian noise with standard deviation $\sigma^2=1$.
Based on the signal in the previous example, we add one small high frequency component on the phase function such that this high frequency part 
cannot be captured during the iteration. Moreover, 
$a_0$ and $a_1$ are not exactly sparse over the Fourier basis in
the $\theta$-space. We also add a white noise to the original signal to 
make it even more challenging to decompose. 

In this example, when the number of sample points is 120, our method can give 92 successful recoveries in 100 independent tests. 
Fig. \ref{error-sparse-appr} gives one of the successful recoveries 
obtained by our algorithm. Due to the truncation error and the noise, 
the error becomes much larger than that in the previous example. But all 
the errors are comparable with the magnitude of the truncation error and 
noise, which shows that our method has good stability even for signals 
with rare samples. When the number of samples is reduced to 80, the 
recovery rate drops to 40 out of 100.
\begin{figure}

    \begin{center}

\includegraphics[width=0.45\textwidth]{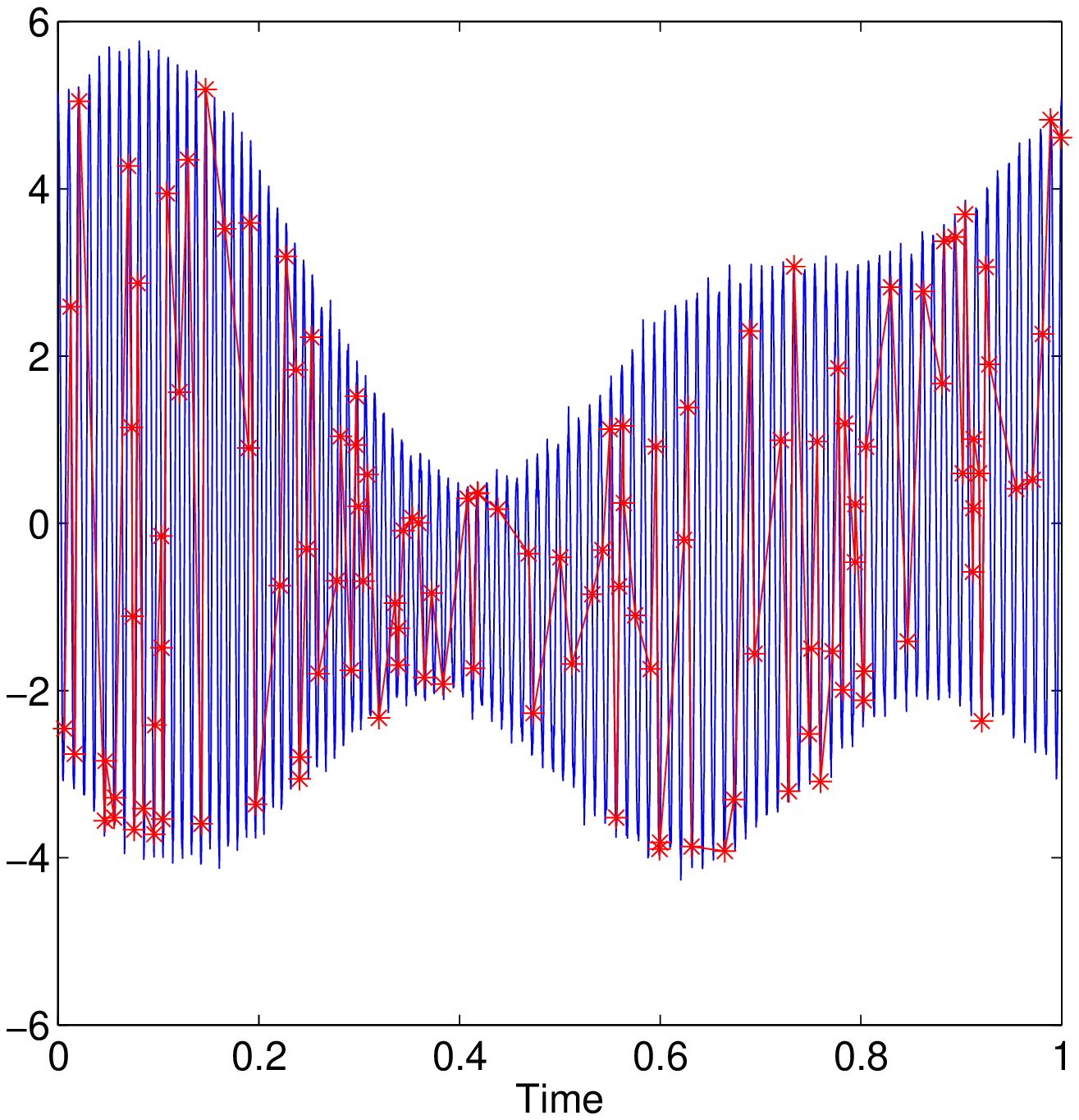}
\includegraphics[width=0.45\textwidth]{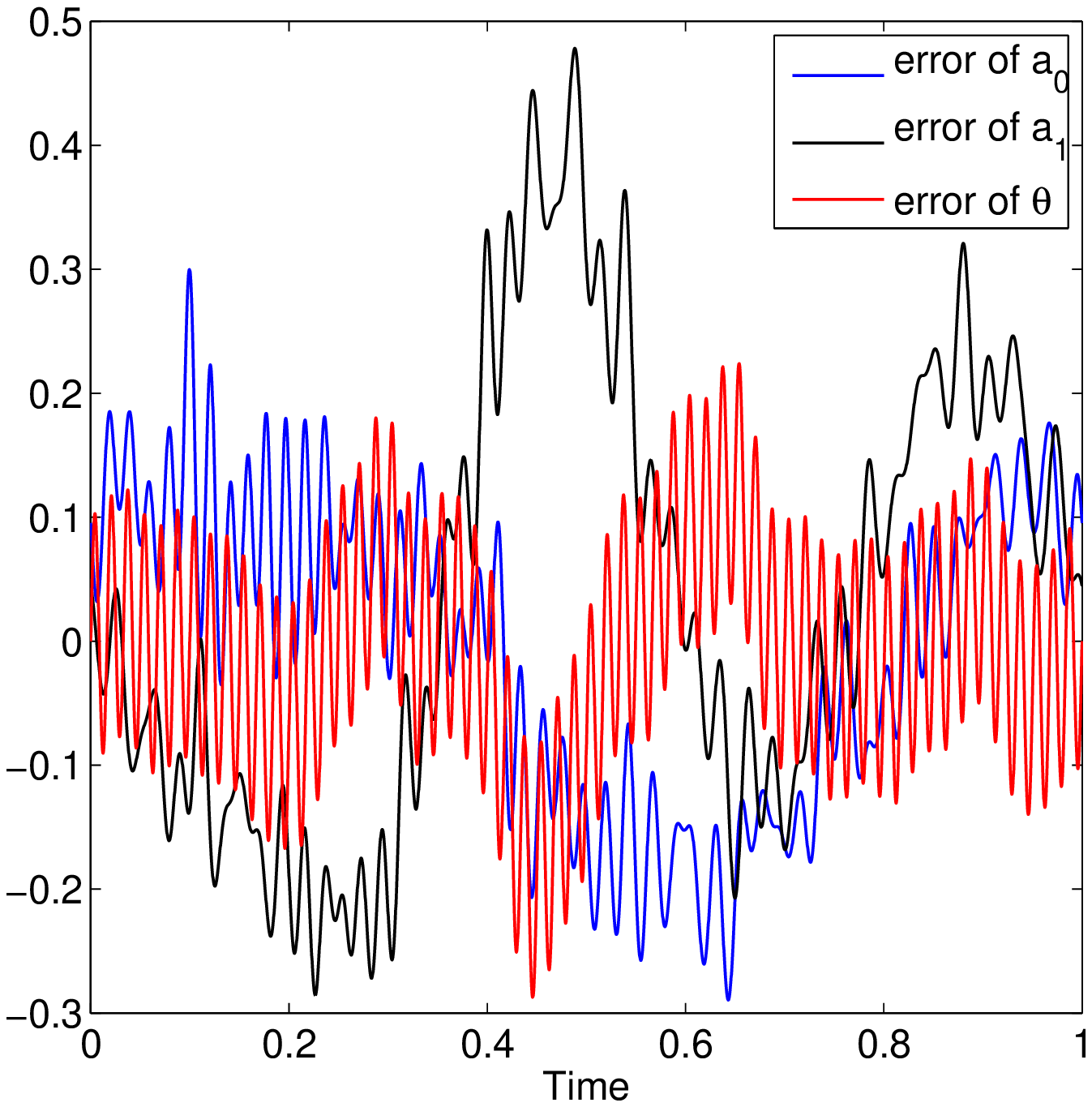}
     \end{center}
    \caption{  \label{error-sparse-appr} Left:Original signal (blue) and the sample points (red) in Ex 3; Right: Errors of $a_0$, $a_1$ and $\theta$.}
\end{figure}

\section{Concluding remarks}
\label{section-conclusion}

In this paper, we analysed the convergence of the data-driven 
time-frequency analysis method proposed in \cite{HS12}.
First, we considered the case when the number of sample points is 
large enough. We proved that the algorithm we developed would converge 
to the exact decomposition if the signal has an intrinsic sparsity 
structure in the coordinate determined by the phase function. We 
also proved the convergence of our method with an approximate 
decomposition when the signal does not have an exact sparse structure 
but its spectral coefficients have a fast decay.

We also considered the more challenging case when only a few number of 
samples are given which do not resolve the original signal accurately. 
In this case, we need to solve a $l^1$ minimization problem which is
computationally more expensive.
We proved the stability and convergence of our method by using
some results developed in compressive sensing.
As in compressive sensing, the convergence and stability of our method
assumes that certain $S$-restricted isometry condition is satisfied.
We proved that for each fixed step in the iteration, this 
$S$-restricted isometry condition is satisfied with an overwhelming 
probability if the sample points are selected at random.

We presented numerical evidence to support our theoretical results. Our numerical results confirmed the theoretical results in all cases that we considered. 

We are currently working on the convergence of the data-driven
time-frequency analysis method  for non-periodic signals. Our extensive
numerical results seem to indicate that our method also converges for 
non-periodic signals. The theoretical analysis for this problem is 
more challenging. We will report the result in a subsequent paper.

\vspace{0.2in}
\noindent
{\bf Acknowledgments.}
This work was in part supported by the AFOSR MURI grant
FA9550-09-1-0613, a DOE grant DE-FG02-06ER25727, and a NSF grant DMS-0908546.  
The research of Dr. Z. Shi was in part supported by a NSFC Grant 11201257.

\vspace{3mm}
\noindent
\textbf{Appendix A: Error of the envelope functions}

Suppose
\begin{eqnarray}
f(t)=f_{0}(t)+f_{1}(t)\cos\theta
\end{eqnarray}
is the signal we want to decompose.

 Let $a^m=f_{1}\cos\Delta\theta^m,\quad b^m=f_{1}\sin\Delta\theta^m$,
then, we have
\begin{eqnarray}
f=f_{0}+a^m\cos\theta^m-b^m\sin\theta^m.
\end{eqnarray}
 Let $L^m=\frac{\theta^m(T)-\theta^m(0)}{2\pi}$ and $\overline{\theta}^m=\theta^m/(2\pi L^m)$.
 Then, we have 
\begin{eqnarray}
\label{exp-f-theta}
f=f_{0}+a^m\cos2\pi L^m\ot^m-b^m\sin2\pi L^m\ot^m.
\end{eqnarray}
 Define the Fourier transform in $\ot$-space as: 
\begin{eqnarray}
\widehat{f}_{\theta^m}=\int_{0}^{1}f(t)e^{-i2\pi k\ot^m}d\ot^m .
\end{eqnarray}
Applying Fourier transform to both sides of \myref{exp-f-theta}, we have
\begin{eqnarray}
\widehat{f}_{\theta^m}(k)=\widehat{f}_{0,\theta^m}(k)+\frac{1}{2}\left(\widehat{a}^m_{\theta^m}(k+L^m)+\widehat{a}^m_{\theta^m}(k-L)\right)
-\frac{i}{2}\left(\widehat{b}^m_{\theta^m}(k+L^m)-\widehat{b}^m_{\theta^m}(k-L^m)\right)\label{fft-f}.\quad\quad
\end{eqnarray}
Then, we get 
\begin{eqnarray*}
  \widehat{a}^m_{\theta^m}(k)-i\widehat{b}^m_{\theta^m}(k)&=&2\widehat{f}_{\theta^m}(k-L^m)-2\widehat{f}_{0,\theta^m}(k-L^m)-\widehat{a}^m_{\theta^m}(k-2L^m)-i\widehat{b}^m_{\theta^m}(k-2L^m),\\
\widehat{a}^m_{\theta^m}(k)+i\widehat{b}^m_{\theta^m}(k)&=&2\widehat{f}_{\theta^m}(k+L^m)-2\widehat{f}_{0,\theta^m}(k+L^m)-\widehat{a}^m_{\theta^m}(k+2L^m)+i\widehat{b}^m_{\theta^m}(k+2L^m).
\end{eqnarray*}
It is easy to solve for 
$\widehat{a}^m_{\theta^m}$ and $\widehat{b}^m_{\theta^m}$ to obtain:
\begin{eqnarray}
\widehat{a}^m_{\theta^m}(k) & = & \widehat{f}_{\theta^m}(k+L^m)+\widehat{f}_{\theta^m}(k-L^m)
-\left[\widehat{f}_{0,\theta^m}(k+L^m)+\widehat{f}_{0,\theta^m}(k-L^m) \right.\nonumber \\ &  & \hspace{0mm}
\left.+\frac{1}{2}\left(\widehat{a}^m_{\theta^m}(k+2L^m)+\widehat{a}^m_{\theta^m}(k-2L^m)\right)
-\frac{i}{2}\left(\widehat{b}^m_{\theta^m}(k+2L^m)-\widehat{b}^m_{\theta^m}(k-2L^m)\right)\right],\label{exp-a-exact}\quad\quad \\
\widehat{b}^m_{\theta^m}(k) & = & -i\left(\widehat{f}_{\theta^m}(k+L^m)-\widehat{f}_{\theta^m}(k-L^m)\right)+i\left[\widehat{f}_{0,\theta^m}(k+L^m)
-\widehat{f}_{0,\theta^m}(k-L^m)
\right.\nonumber \\
 &  & \hspace{0mm}\left.+\frac{1}{2}\left(\widehat{a}^m_{\theta^m}(k+2L^m)-\widehat{a}^m_{\theta^m}(k-2L^m)\right)
-\frac{i}{2}\left(\widehat{b}^m_{\theta^m}(k+2L^m)+\widehat{b}^m_{\theta^m}(k-2L^m)\right)\right].\label{exp-b-exact}
\end{eqnarray}
 In our algorithm, $\mathcal{F}_{\theta^m}(\widetilde{a}^m)$ and $\mathcal{F}_{\theta^m}(\widetilde{b}^m)$ are approximated
in the following way: 
\begin{eqnarray}
\widehat{\widetilde{a}}^m_{\theta^m}(k) & = & \left\{ \begin{array}{cc}
\widehat{f}_{\theta^m}(k+L^m)+\widehat{f}_{\theta^m}(k-L^m), & -L^m/2\le k\le L^m/2,\\
0, & otherwise.
\end{array}\right.\\
\widehat{\widetilde{b}}^m_{\theta^m}(k) & = & \left\{ \begin{array}{cc}
-i(\widehat{f}_{\theta^m}(k+L^m)-\widehat{f}_{\theta^m}(k-L^m)), & -L^m/2\le k\le L^m/2,\\
0, & otherwise.
\end{array}\right.
\end{eqnarray}
 Then, we can get the error of the approximation in the spectral space:
\begin{eqnarray*}
\widehat{\Delta a}^m_{\theta^m}(k)=\left\{ \begin{array}{cl}
-\left[\widehat{f}_{0,\theta^m}(k+L^m)+\widehat{f}_{0,\theta^m}(k-L^m)+\frac{1}{2}\left(\widehat{a}^m_{\theta^m}(k+2L^m)+\widehat{a}^m_{\theta^m}(k-2L^m)\right)\right.\\
\hspace{1cm}\left.-\frac{i}{2}\left(\widehat{b}^m_{\theta^m}(k+2L^m)-\widehat{b}^m_{\theta^m}(k-2L^m)\right)\right], & |k|\le L^m/2,\\
\widehat{a}^m_{\theta^m}(k), & |k|>L^m/2.
\end{array}\right.
\end{eqnarray*}
\begin{eqnarray*}
\widehat{\Delta b}^m_{\theta^m}(k)=\left\{ \begin{array}{cl}
i\left[\widehat{f}_{0,\theta^m}(k+L^m)-\widehat{f}_{0,\theta^m}(k-L^m)+\frac{1}{2}\left(\widehat{a}^m_{\theta^m}(k+2L^m)-\widehat{a}^m_{\theta^m}(k-2L^m)\right)\right.\\
\left.
-\frac{i}{2}\left(\widehat{b}^m_{\theta^m}(k+2L^m)+\widehat{b}^m_{\theta^m}(k-2L^m)\right)\right], & |k|\le L^m/2,\\
\widehat{b}^m_{\theta^m}(k), & |k|>L^m/2.
\end{array}\right.
\end{eqnarray*}
Thus, we have the following inequality for the $l^1$ norm of the error in the spectral space:
\begin{eqnarray}
\label{exp-error-a}
 |\Delta a^m|&\le & \|\widehat{\Delta a}^m_{\theta^m}\|_{1}\nonumber \\
 & \le &2\sum_{\frac{L^m}{2}<k<\frac{3}{2}L^m}\left|\widehat{f}_{0,\theta^m}(k)\right|+\sum_{\frac{3}{2}L^m<k<\frac{5}{2}L^m}\left(\left|
\widehat{a}^m_{\theta^m}(k)\right|+\left|\widehat{b}^m_{\theta^m}(k)\right|\right)+\sum_{|k|>\frac{L^m}{2}}\left|\widehat{a}^m_{\theta^m}(k)\right|.\quad\quad\quad
\end{eqnarray}
Similarly,  we get
\begin{eqnarray}
\label{exp-error-b}
  |\Delta b^m|\le 2\sum_{\frac{L^m}{2}<k<\frac{3}{2}L^m}\left|\widehat{f}_{0,\theta^m}(k)\right|+\sum_{\frac{3}{2}L^m<k<\frac{5}{2}L^m}\left(\left|
\widehat{a}^m_{\theta^m}(k)\right|+\left|\widehat{b}^m_{\theta^m}(k)\right|\right)+\sum_{|k|>\frac{L^m}{2}}\left|\widehat{b}^m_{\theta^m}(k)\right|.\quad\quad\quad
\end{eqnarray}

In the above derivation, we assume that the Fourier transform of $f$ in $\theta^m$-space can be calculated exactly. If only approximate
Fourier transform is available, denoted as $\widehat{\widetilde{f}}_{\theta^m}$, such as the signal with sparse samples we discussed in Section \ref{section-sparse}, there would 
be an extra term in the estimates of $\Delta a^m$ and $\Delta b^m$,
\begin{eqnarray}
\label{exp-error-a-ss}
 |\Delta a^m|
  & \le &2\sum_{\frac{L^m}{2}<k<\frac{3}{2}L^m}\left|\widehat{f}_{0,\theta^m}(k)\right|+\sum_{\frac{3}{2}L^m<k<\frac{5}{2}L^m}\left(\left|
\widehat{a}^m_{\theta^m}(k)\right|+\left|\widehat{b}^m_{\theta^m}(k)\right|\right)\nonumber\\
&&+\sum_{|k|>\frac{L^m}{2}}\left|\widehat{a}^m_{\theta^m}(k)\right|+2\sum_{\frac{L^m}{2}<k<\frac{3}{2}L^m}\left|\widehat{f}_{\theta^m}(k)-
\widehat{\widetilde{f}}_{\theta^m}(k)\right|,\\
\label{exp-error-b-ss}
 |\Delta b^m|
  & \le &2\sum_{\frac{L^m}{2}<k<\frac{3}{2}L^m}\left|\widehat{f}_{0,\theta^m}(k)\right|+\sum_{\frac{3}{2}L^m<k<\frac{5}{2}L^m}\left(\left|
\widehat{a}^m_{\theta^m}(k)\right|+\left|\widehat{b}^m_{\theta^m}(k)\right|\right)\nonumber\\
&&+\sum_{|k|>\frac{L^m}{2}}\left|\widehat{b}^m_{\theta^m}(k)\right|+2\sum_{\frac{L^m}{2}<k<\frac{3}{2}L^m}\left|\widehat{f}_{\theta^m}(k)-
\widehat{\widetilde{f}}_{\theta^m}(k)\right| .
\end{eqnarray}

\vspace{3mm}
\noindent
\textbf{Appendix B: Estimates of $\widehat{f}_{0,\theta^m}(\omega)$, $\widehat{a}^m_{\theta^m}(\omega)$ 
and $\widehat{b}^m_{\theta^m}(\omega)$in Theorem \ref{theorem-sparse}.}

We first estimate $f_0$. We have
\begin{eqnarray}
|\widehat{f}_{0,\theta^m}(\omega)| & = & \left|\int_{0}^{1} f_{0}(t)e^{-i2\pi\omega\ot^m}d\ot^m\right|\nonumber \\
& = & \int_0^1\sum_{|k|\le M_{1}}\widehat{f}_{0,\theta}(k)e^{i2\pi (k\ot-\omega\ot^m)}d\ot^m\nonumber \\
 & = & \sum_{|k|\le M_{1}}\widehat{f}_{0,\theta}(k)\int_{0}^{1}e^{i2\pi (k\ot-\omega\ot^m)}d\ot^m\nonumber\\
 & = & \left|\sum_{|k|\le M_{1}}\widehat{f}_{0,\theta}(k)\int_{0}^{1}e^{i2\pi(\al k-\omega)\ot^m}
e^{ik\dt^m/L}d\ot^m\right|,
\end{eqnarray}
where $\al=L^m/L$. In the last equality, we have used the fact 
that $\theta=2\pi L \ot,\; \theta^m=2\pi L^m \ot^m$ and $\theta=\theta^m+\Delta \theta^m$.

Using Lemma \ref{lemma-fft}, we obtain for any $|\omega|>L/2$ that
\begin{eqnarray}
\label{est-a0-0}
|\widehat{f}_{0,\theta^m}(\omega)| & \le & \left|\sum_{|k|\le M_{1}}\widehat{f}_{0,\theta}(k)\int_{0}^{1}e^{i2\pi(\al k-\omega)\ot^m}
e^{ik\dt^m/L}d\ot^m\right|\nonumber \\
 & \le & C_{0}\sum_{|k|\le M_{1}}\frac{QM_{0}^{n}}{|\omega-\al k|^{n}}\sum_{j=1}^{n}\left|\frac{k}{L}\right|^{j}(2\pi M_{0})^{-j}\|\mathcal{F}_{\theta^m}[(\Delta \theta^m)']\|_{1}^{j}\nonumber \\
 & \le & 2C_{0}Q\left(\frac{|\omega|}{2}\right)^{-n}M_{0}^{n}M_{1}\sum_{j=1}^{n}\left(M_{1}\gamma/L\right)^{j},
\end{eqnarray}
 where 
\begin{eqnarray}
Q=\frac{P\left(z,n\right)}{\left(\min(\ot^m)'\right)^{n}},\quad z=\frac{\|\mathcal{F}[(\ot^m)']\|_{1}}{\min(\ot^m)'},\quad  \gamma=\frac{\|\mathcal{F}[(\Delta \theta^m)']\|_{1}}{2\pi M_{0}}.
\end{eqnarray}
In the above derivation, we need to assume that $L\ge 4M_1$ such that $|\omega-\al k|\ge |\omega|/2$ for all $|\omega|\ge L/2$ and $|k|\le M_1$.

If we further assume that $\gamma \le 1/4$, we have
\begin{eqnarray}
|\widehat{f}_{0,\theta^m}(\omega)| 
&\le &  C_{0}Q\left(\frac{|\omega|}{2}\right)^{-n}M_{0}^{n}M_{1}\gamma .
\end{eqnarray}

Next, we estimate $\widehat{a}^m_{\theta^m}$. The method of analysis
is similar to the previous one, however the derivation is a little more complicated. We proceed as follows:
\begin{eqnarray}
\label{est-a-0}
|\widehat{a}^m_{\theta^m}(\omega)| & = & \left|\int_{0}^{1}f_{1}(t)\cos\dt^m(t) e^{-i2\pi\omega\ot^m}d\ot^m\right|\nonumber\\
 & \le & \frac{1}{2}\left|\int_{0}^{1}\sum_{|k|\le M_1}f_{1,\theta}(k)e^{i2\pi k\ot} (e^{i\dt}+e^{-i\dt})e^{-i2\pi\omega\ot^m}d\ot^m\right|\nonumber\\
 & \le & \frac{1}{2}\left|\sum_{|k|\le M_1}f_{1,\theta}(k)\int_{0}^{1}e^{i2\pi(\al k-\omega)\ot^m}e^{i(k+L)\dt/L}d\ot^m\right|\nonumber\\
 &  & +\frac{1}{2}\left|\sum_{|k|\le M_1}f_{1,\theta}(k)\int_{0}^{1}e^{i2\pi(\al k-\omega)\ot^m}e^{i(k-L)\dt/L}d\ot^m\right|.
\end{eqnarray}
For the first term in the above inequality, we have that for any $|\omega|>L/2$,
\begin{eqnarray}
\label{est-a-1}
 &  & \left|\sum_{|k|\le M_1}f_{1,\theta}(k)\int_{0}^{1}e^{i2\pi(\al k-\omega)\ot^m}e^{i(k+L)\dt/L}d\ot^m\right|\nonumber \\
 & \le & C_{0}Q\sum_{|k|\le M_{1}}\frac{M_{0}^{n}}{|\omega-\al k|^{n}}\sum_{j=1}^{n}\left|1+\frac{k}{L}\right|^{j}\gamma^{j}\nonumber \\
 & \le & C_{0}Q\left(\frac{|\omega|}{2}\right)^{-n}M_{0}^{n}\sum_{j=1}^{n}2^{j-1}\gamma^{j}\sum_{|k|\le M_{1}}\left(1+\left|\frac{k}{L}\right|^{j}\right)\nonumber \\
 & \le & 4C_{0}Q\left(\frac{|\omega|}{2}\right)^{-n}M_{0}^{n}(2M_{1}+1)\gamma.
\end{eqnarray}
Here we also assume that $L\ge 4M_1,\; \gamma\le 1/4$. The definition of $Q$ and $\gamma$ can be found in \myref{def-Q}.

For the second term in \myref{est-a-0}, we can get the same bound for $|\omega|\ge L/2$,
\begin{eqnarray}
\label{est-a-2}
\left|\sum_{|k|\le M_1}f_{1,\theta}(k)\int_{0}^{1}e^{i2\pi(\al k-\omega)\ot^m}e^{i(k-L)\dt/L}d\ot^m\right| \le 4C_{0}Q\left(\frac{|\omega|}{2}\right)^{-n}M_{0}^{n}(2M_{1}+1)\gamma.
\end{eqnarray}
By combining \myref{est-a-0},\myref{est-a-1} and \myref{est-a-2}, we obtain a complete control of $\ho{a}$,
\begin{eqnarray}
|\widehat{a}^m_{\theta^m}(\omega)| & \le & 4C_{0}Q\left(\frac{|\omega|}{2}\right)^{-n}M_{0}^{n}(2M_{1}+1)\gamma, \quad \forall |\omega|\ge L/2.
\end{eqnarray}
Similarly, we can estimate $\widehat{b}^m_{\theta^m}$ by the same upper bound,
\begin{eqnarray}
|\widehat{b}^m_{\theta^m}(\omega)| & \le & 4C_{0}Q\left(\frac{|\omega|}{2}\right)^{-n}M_{0}^{n}(2M_{1}+1)\gamma, \quad \forall |\omega|\ge L/2.
\end{eqnarray}

\vspace{3mm}
\newpage
\noindent
\textbf{Appendix C: Proof of Lemma \ref{lemma-integral-dtheta}}
\begin{proof}
Since $e^{i2\pi k\phi}$ is a periodic
function over $[0,1]$, it can be represented by Fourier series:
\begin{eqnarray}
e^{i2\pi k\phi(t)}=\sum_{l=-\infty}^{+\infty}d_{l}e^{i2\pi lt},\quad t\in[0,1],
\end{eqnarray}
 where $d_{l}=\int_{0}^{1}e^{i2\pi k\phi(t)}e^{-i2\pi lt}dt$.
By assumption, we have $\phi'(t)\in V_{M_{0}}$. Thus,  we get 
\begin{eqnarray}
\phi'(t)=\sum_{j=-M_{0}}^{M_{0}}c_{j}e^{i2\pi jt},\quad t\in[0,1],
\end{eqnarray}
 where $c_{j}=\int_{0}^{1}\ot'(t)e^{-i2\pi jt}dt$.

Then, we have 
\begin{eqnarray}
 &  & \frac{1}{L}\sum_{m=0}^{L-1}\phi'(t_{m})e^{i2\pi k\phi(t_{m})}\nonumber \\
 & = & \frac{1}{L}\sum_{m=0}^{L-1}\sum_{j=-M_{0}}^{M_{0}}\sum_{l=-\infty}^{+\infty}c_{j}d_{l}e^{i2\pi(l+j)t_{m}}\nonumber \\
 & = & \frac{1}{L}\sum_{j=-M_{0}}^{M_{0}}\sum_{l=-\infty}^{+\infty}c_{j}d_{l}\sum_{m=0}^{L-1}e^{i2\pi(l+j)m/L}\nonumber \\
 & = & \sum_{j=-M_{0}}^{M_{0}}\sum_{p\in\mathbb{Z}}c_{j}d_{pL-j}\nonumber \\
 & = & \sum_{j=-M_{0}}^{M_{0}}c_{j}d_{-j}+\sum_{j=-M_{0}}^{M_{0}}\sum_{p\in\mathbb{Z},p\ne0}c_{j}d_{pL-j}\nonumber \\
 & = & \int_{0}^{1}\ot'(t)e^{i2\pi k\phi(t)}dt+\sum_{j=-M_{0}}^{M_{0}}\sum_{p\in\mathbb{Z},p\ne0}c_{j}d_{pL-j}\nonumber \\
 & = & \sum_{j=-M_{0}}^{M_{0}}\sum_{p\in\mathbb{Z},p\ne0}c_{j}d_{pL-j}.
\end{eqnarray}

Using integration by parts, we have 
\begin{eqnarray}
|d_{l}| & = & |\int_{0}^{1}e^{i2\pi k\phi}e^{-i2\pi lt}dt|\nonumber \\
 & = & \frac{1}{|l|^{n}}\left|\int_{0}^{1}\left(\frac{d^{n}}{dt^{n}}e^{i2\pi k\phi}\right)e^{-i2\pi lt}dt\right|\nonumber \\
 & \le & \frac{1}{|l|^{n}}\int_{0}^{1}\left|\left(\frac{d^{n}}{dt^{n}}e^{i2\pi k\phi}\right)\right|dt\nonumber \\
 & \le & \frac{1}{|l|^{n}}\max_{t}\left|\left(\frac{d^{n}}{dt^{n}}e^{i2\pi k\phi}\right)\right|.
\end{eqnarray}
Using the inequality \myref{control-g} in the proof of Lemma \ref{lemma-fft},
and by a direct calculation, we can show that for any $n>0$, 
there exists $C(n)>0$, such that 
\begin{eqnarray}
\max_{t}\left|\left(\frac{d^{n}}{dt^{n}}e^{i2\pi k\phi}\right)\right| & \le & C(n)\sum_{j=1}^{n}|k|^{j}M_{0}^{n-j}\|\widehat{\phi'}\|_{1}^{j}
=C(n)|k|M_{0}^{n-1}\|\widehat{\phi'}\|_{1}\frac{\frac{|k|^{n}}{M_{0}^{n}}\|\widehat{\phi'}\|_{1}^{n}-1}{\frac{|k|}{M_{0}}\|\widehat{\phi'}\|_{1}-1}\nonumber \\
 & \le & \left\{ \begin{array}{cc}
2C(n)|k|^{n}\|\widehat{\phi'}\|_{1}^{n}, & \frac{|k|}{M_{0}}\|\widehat{\phi'}\|_{1}>2,\\
2C(n)(2M_{0})^{n}, & \frac{|k|}{M_{0}}\|\widehat{\phi'}\|_{1}\le2.
\end{array}\right.
\end{eqnarray}
As a result, we obtain
\begin{eqnarray}
|d_{l}|\le\left\{ \begin{array}{cc}
2C(n)\left|\frac{k\|\widehat{\phi'}\|_{1}}{l}\right|^{n}, & |k|\|\widehat{\phi'}\|_{1}>2M_{0},\\
2C(n)\left|\frac{2M_{0}}{l}\right|^{n}, & |k|\|\widehat{\phi'}\|_{1}\le2M_{0},
\end{array}\right.
\end{eqnarray}
Finally, we derive the following estimate 
\begin{eqnarray}
 &  & \left|\sum_{j=-M_{0}}^{M_{0}}\sum_{p\in\mathbb{Z},p\ne0}c_{j}d_{pL-j}\right|\nonumber \\
 & \le & \sum_{p\in\mathbb{Z},p\ne0}\sum_{j=-M_{0}}^{M_{0}}|c_{j}||d_{pL-j}|\nonumber \\
 & \le & 2\sum_{j=-M_{0}}^{M_{0}}|c_{j}|\sum_{p=1}^{+\infty}\max_{j}|d_{pL-j}|\nonumber \\
 & \le & 4C(n)\|\widehat{\phi'}\|_{1}\sum_{p=1}^{+\infty}\max\left(\left|\frac{k\|\widehat{\phi'}\|_{1}}{pL-M_{0}}\right|^{n},\left|\frac{2M_{0}}{pL-M_{0}}\right|^{n}\right)\nonumber \\
 & \le & 4C(n)\|\widehat{\phi'}\|_{1}\max\left(\left|\frac{k\|\widehat{\phi'}\|_{1}}{L}\right|^{n},\left|\frac{2M_{0}}{L}\right|^{n}\right)\sum_{p=1}^{+\infty}(p-M_{0}/L)^{-n}\nonumber \\
 & \le & 4\left(1-M_{0}/L\right)^{-n+1}\frac{C(n)}{n-1}\|\widehat{\phi'}\|_{1}\max\left(\left|\frac{k\|\widehat{\phi'}\|_{1}}{L}\right|^{n},\left|\frac{2M_{0}}{L}\right|^{n}\right).
\end{eqnarray}
\end{proof}

\end{document}